\documentclass[12pt,reqno]{amsart}

\usepackage[a4paper, margin=1in]{geometry}

\usepackage {color}
\usepackage{amsmath,amsthm,amssymb}
\usepackage{epsfig}
\usepackage{graphicx}
\usepackage{amssymb}
\usepackage{latexsym}
\usepackage{pdfpages}

\usepackage{ulem} 

\usepackage{ifpdf}
\newcommand{\res}{\!\!\mathop{\hbox{
                                \vrule height 7pt width .5pt depth 0pt
                                \vrule height .5pt width 6pt depth 0pt}}
                                \nolimits}

\def\z{{\bf z}}
\def\a{{\bf a}}

\ifpdf 
  \usepackage[hidelinks]{hyperref}
\else 
\fi 

 \usepackage[usenames,dvipsnames]{pstricks}
 \usepackage{epsfig}
 \usepackage{pst-grad} 
 \usepackage{pst-plot} 

\usepackage{color}

\newtheorem{theorem}{Theorem}[section]
\newtheorem{lemma}[theorem]{Lemma}
\newtheorem{definition}[theorem]{Definition}
\newtheorem{proposition}[theorem]{Proposition}

\newtheorem{remark}[theorem]{Remark}
\newtheorem{example}[theorem]{Example}

\newtheorem*{theorem*}{\it Theorem}

\newtheorem{assumption}{Assumption}

\def\vint_#1{\mathchoice%
          {\mathop{\kern 0.2em\vrule width 0.6em height 0.69678ex depth -0.58065ex
                  \kern -0.8em \intop}\nolimits_{\kern -0.4em#1}}%
          {\mathop{\kern 0.1em\vrule width 0.5em height 0.69678ex depth -0.60387ex
                  \kern -0.6em \intop}\nolimits_{#1}}%
          {\mathop{\kern 0.1em\vrule width 0.5em height 0.69678ex
              depth -0.60387ex
                  \kern -0.6em \intop}\nolimits_{#1}}%
          {\mathop{\kern 0.1em\vrule width 0.5em height 0.69678ex depth -0.60387ex
                  \kern -0.6em \intop}\nolimits_{#1}}}
\def\vintslides_#1{\mathchoice%
          {\mathop{\kern 0.1em\vrule width 0.5em height 0.697ex depth -0.581ex
                  \kern -0.6em \intop}\nolimits_{\kern -0.4em#1}}%
          {\mathop{\kern 0.1em\vrule width 0.3em height 0.697ex depth -0.604ex
                  \kern -0.4em \intop}\nolimits_{#1}}%
          {\mathop{\kern 0.1em\vrule width 0.3em height 0.697ex depth -0.604ex
                  \kern -0.4em \intop}\nolimits_{#1}}%
          {\mathop{\kern 0.1em\vrule width 0.3em height 0.697ex depth -0.604ex
                  \kern -0.4em \intop}\nolimits_{#1}}}

\def\R{\mathbb R}
\def\N{\mathbb N}

\numberwithin{equation}{section}

\def\1{\raisebox{2pt}{\rm{$\chi$}}}

\definecolor{violet(ryb)}{rgb}{0.53, 0.0, 0.69}
\usepackage{mathtools}
\mathtoolsset{showonlyrefs}

\begin{document}  

\title[Evolution problems of Leray-Lions  type ]{\bf   Evolution  problems of Leray-Lions  type  with  nonhomogeneous  Neumann boundary conditions in metric random walk spaces}

\author[J. M. Maz\'on, M. Solera and J. Toledo]{Jos\'e M. Maz\'on, Marcos Solera and Juli\'an Toledo}

\address{J. M. Maz\'{o}n, M. Solera and J. Toledo: Departamento de An\'{a}lisis Matem\'atico,
Universitat de Val\`encia, Valencia, Spain.
 {\tt mazon@uv.es}, {\tt marcos.solera@uv.es} and {\tt toledojj@uv.es }
  }


\date{}

\begin{abstract}
 In this paper we study evolution problems of Leray-Lions type with nonhomogeneous Neumann boundary conditions in the framework of metric random walk spaces.  This covers cases with the $p$-Laplacian operator in weighted discrete graphs and nonlocal operators with nonsingular kernel in $\R^N$.
\end{abstract}

\maketitle


{ \renewcommand\contentsname{Contents}
\setcounter{tocdepth}{3}
{\small \tableofcontents}
}

\section{Introduction and preliminaries}

  A  metric random walk space $[X,d,m]$ is a Polish metric space  $(X,d)$ together with a family  $m = (m_x)_{x \in X}$ of probability measures that encode the jumps of a Markov chain. Important examples of metric random walk spaces are: locally finite weighted graphs, finite Markov chains and $[\R^N, d, m^J]$ with $d$ the Euclidean distance and
$$m^J_x(A) :=  \int_A J(x - y) d\mathcal{L}^N(y) \quad \hbox{ for every Borel set } A \subset  \R^N ,$$
where $J:\R^N\to[0,+\infty[$ is a measurable, nonnegative and radially symmetric
function with $\int J d\mathcal{L}^N=1$. See Section~\ref{semrw} for more details.

   The aim of this paper is to study $p$-Laplacian type evolution problems like the  one given in the following reference model:
\begin{equation}\label{refemod}u_t(t,x) = \displaystyle\int_{\Omega\cup \partial_m\Omega}   \vert u(y)-u(x)\vert^{p-2}(u(y) - u(x)) dm_x(y), \quad
    x\in  \Omega,\ 0<t<T,
\end{equation}
 with
$$\hbox{
\it nonhomogeneous Neumann boundary conditions,}
$$
 where $\Omega \subset X$ and
  $\partial_m\Omega=\{ x\in X\setminus \Omega : m_x(\Omega)>0 \}$  is {\it the $m$-boundary of $\Omega$}. This reference model can be regarded as the nonlocal counterpart to the classical evolution problem
$$\left\{\begin{array}{c}
\displaystyle u_t=\hbox{div}(|\nabla u|^{p-2}\nabla u),\quad x\in  U,\ 0<t<T,\\ \\ \displaystyle -|\nabla u|^{ p-2}\nabla u\cdot \eta=\varphi, \quad x\in  \partial U,\ 0<t<T,
\end{array}\right.
$$
where $U$ is a bounded smooth domain in $\mathbb{R}^n$, and $\eta$ is the outer normal   vector to $\partial U$.

\medskip

 Nonlocal diffusion problems of $p$-Laplacian type with homogeneous Neumann boundary conditions have been studied (see Examples \ref{graphs101} and \ref{graphs1} for the notation) in $[\R^N, d, m^J]$ (see, for example, \cite{BLGneu}, \cite{ElLibro}) and in graphs $[V(G), d_G, (m^G_x)]$ (see, for example, the work of   Hafiene,  Fadili   and  Elmoataz \cite{fadili01}) with the   formulation
\begin{equation}\label{sinper}u_t(t,x) = \displaystyle\int_{\Omega}   \vert u(y)-u(x)\vert^{p-2}(u(y) - u(x))dm_x(y), \quad
    x\in  \Omega,\ 0<t<T.
    \end{equation}
      Here, the homogeneous Neumann boundary conditions are understood in the sense that the jumps of the Markov chain are restricted to staying in $\Omega$ (which is consistent with what happens in the classical local model). See  also \cite[Example 2.3]{MST0} for the linear case,  i.e., $p=2$, in metric random walk spaces.

  The linear case with nonhomogeneous boundary conditions has been addressed by different authors. For example, Cortazar et~al.~in \cite{CERW1} present this case, for   non-singular kernels, as a perturbation of Problem~\eqref{sinper} ($p=2$). Moreover, in \cite{GL1},  Gunzburger and Lehoucq  develop a nonlocal vector calculus    with applications to linear nonlocal problems in which the nonlocal Neumann boundary condition  considered is,   written in the context of metric random walk spaces,
\begin{equation}\label{NeumGL}
-\int_{\Omega_m}   (u(y) - u(x)) dm_x(y) =\varphi(x), \quad x \in \partial_m\Omega,
\end{equation}
where $\Omega_m=\Omega\cup\partial_m\Omega$. Another interesting approach is proposed by Dipierro, Ros-Oton and Valdinoci in~\cite{DR-OV} for the particular case of the fractional Laplacian diffusion (although the idea can be used for other kernels) with the following Neumann boundary condition,    that we rewrite in the context of metric random walk spaces,
\begin{equation}\label{Neumann2intdro}
 -\int_\Omega     (u(x) - u(y)) dm_x(y)=\varphi(x), \quad x \in {\partial_m\Omega},
\end{equation}
 or, alternatively, if one prefers a normalized   boundary condition with respect to the underlying probability measure induced by the jump process under consideration,
$$
-\frac{1}{m_x(\Omega)}\int_\Omega     (u(x) - u(y)) dm_x(y)=\varphi(x), \quad x \in {\partial_m\Omega}.
$$
 Therefore, as remarked in~\cite{DR-OV},   when a particle exits $\Omega$ to a point $x\in\partial_m\Omega$, the mass $u(x)-\varphi(x)$ {\it comes back} into $\Omega$ following $\frac{1}{m_x(\Omega)}m_x$:
$$\frac{1}{m_x(\Omega)}\int_\Omega u(y)dm_x(y)=u(x)-\varphi(x),\quad x\in \partial_m \Omega .$$
A similar probabilistic interpretation can be given for the Neumann boundary condition~\eqref{NeumGL} but involving all of $\Omega_m$.
Anyhow, observe that the formulations~\eqref{NeumGL} and~\eqref{Neumann2intdro} have an important difference in their definition regarding the domain of integration.

 The approaches in \cite{CERW1} and \cite{DR-OV} have been  unified in~\cite{Abata}.
  Conditions like~\eqref{Neumann2intdro} were also introduced for graphs by  Chung and Yau in~\cite{Chy0} and \cite{Chy1} (see also \cite{HH1} and \cite{HHW}) for the study of the eigenvalues of the graph Laplacian operator. Namely,  let $G = (V(G), E(G))$ be a finite weighted discrete connected graph (see Example \ref{graphs1}) and $\Omega \subset  V(G)$ be a set of vertices of  $G$, their work comprises the study of the eigenvalues of the graph Laplacian operator given by
$$
\frac{1}{d_x} \sum_{y \in V}  (u(y)-u(x)) w_{xy}, \quad x \in \Omega,$$
  under the following  Neumann boundary condition:
$$ -\frac{1}{d_x} \sum_{y \in \Omega} (u(y)-u(x)) w_{xy}=0,   \quad x \in \delta\Omega,
$$
  where $\delta\Omega$ is the vertex boundary  of  $\Omega$, defined as $$\delta\Omega := \{ x \in V(G) \setminus \Omega \ : \ \exists y \in \Omega \ s.t. \ y \sim x \}, $$
  which coincides with $\partial_{m^G}\Omega$.

 We study the above formulations for the nonlinear case with nonhomoge\-neous boundary conditions in the general framework of metric random walk spaces. The main tool used for this is Nonlinear Semigroup Theory.  We will consider two types of nonhomogeneous Neumann boundary conditions, one in the line of the work of Gunzburger and Lehoucq (Problem~\eqref{Neumann4}) and the other following the approach taken by Dipierro, Ros-Oton and Valdinoci (Problem~\eqref{NNeumann4}). For the first type, we will obtain existence and uniqueness of solutions in Theorem~\ref{nsth01}  by assuming that a Poincar\'{e} type inequality, which is satisfied  by most of the important examples, holds.  Before that, in order to apply the necessary results from Nonlinear Semigroup Theory, we solve the corresponding {\it elliptic type problem} in Theorem~\ref{teo27}. For the second type, existence and uniqueness is proved in Theorem~\ref{nsth01N} without assuming that a Poincar\'{e} type inequality holds. We first solve the corresponding {\it elliptic type problem} in Theorem~\ref{CAandRangeCond}, this is done by using monotonicity techniques.  The study of these nonhomogeneous boundary conditions had, to our knowledge, not yet been done, not even for singular kernels or for particular cases covered by the general framework of metric random walk spaces.

\subsection{Metric random walk  spaces}\label{semrw}
Let $(X,d)$ be a  Polish metric  space equipped with its Borel $\sigma$-algebra.
A {\it random walk} $m$ on $X$ is a family of probability measures $m_x$ on $X$, $x \in X$, satisfying the two technical conditions: (i) the measures $m_x$  depend measurably on the point  $x \in X$, i.e., for any Borel set $A$ of $X$ and any Borel set $B$ of $\R$, the set $\{ x \in X \ : \ m_x(A) \in B \}$ is Borel; (ii) each measure $m_x$ has finite first moment, i.e. for some (hence any) $z \in X$, and for any $x \in X$ one has $\int_X d(z,y) dm_x(y) < +\infty$ (see~\cite{O}).

A {\it metric random walk  space} $[X,d,m]$  is a  Polish metric space $(X,d)$ equipped with a  random walk $m$.
  A   Radon measure $\nu$ on $X$ is {\it invariant} for the random walk $m=(m_x)$ if
$$d\nu(x)=\int_{y\in X}d\nu(y)dm_y(x).$$
The measure $\nu$ is said to be {\it reversible} if, moreover, the detailed balance condition $$dm_x(y)d\nu(x)  = dm_y(x)d\nu(y) $$ holds. Under suitable assumptions on the  metric random walk  space $[X,d,m]$, such an invariant and reversible measure $\nu$ exists and is unique. Note that the reversibility condition implies the invariance condition.

\begin{assumption}\label{as1}{\rm
When dealing with a metric random walk space $[X,d,m]$, we will assume that there exists an invariant and reversible measure for the random walk, which we will always denote by $\nu$, such that $m_x\ll \nu\quad\hbox{for all }\ x\in X$. Moreover, we will assume that the metric random walk space together with $\nu$ is $m$-connected (see~\cite{MST0}).}
\end{assumption}

    Important examples of metric random walk spaces are the following:

     \begin{example}\label{graphs101}{\rm
      Consider $(\R^N, d, \mathcal{L}^N)$, where $d$ is the Euclidean distance and $\mathcal{L}^N$ the Lebesgue measure. Let  $J:\R^N\to[0,+\infty[$ be a measurable, nonnegative and radially symmetric
function  verifying $\int_{\R^N}J(z)dz=1$. In $(\R^N, d, \mathcal{L}^N)$ we define the following random walk:
$$m^J_x(A) :=  \int_A J(x - y) d\mathcal{L}^N(y) \quad \hbox{ for every Borel set } A \subset  \R^N \hbox{ and }x\in\R^N.$$
Applying Fubini's Theorem it easy to see that the Lebesgue measure $\mathcal{L}^N$ is an invariant and  reversible measure for this random walk.
}
 \end{example}

 \begin{example}\label{graphs1}{\rm Consider  a  weighted discrete graph $G = (V(G), E(G))$, where each edge $(x,y) \in E(G)$ (we will write $x\sim y$ if $(x,y) \in E(G)$) has a positive weight $w_{xy} = w_{yx}$ assigned. Suppose further that $w_{xy} = 0$ if $(x,y) \not\in E(G)$. We then equip the graph with the standard shortest path graph distance $d_G$, that is, $d_G(x,y)$ is the minimal number of
edges which form a path connecting $x$ and $y$.  Assume that any two vertices are connected, i.e., that the graph is connected. For $x \in V(G)$ we define the weight   at the vertex $x$ as $$d_x:= \sum_{y\sim x} w_{xy} = \sum_{y\in V(G)} w_{xy}.$$
 When $w_{x,y}=1$ for every $(x,y)\in E(G)$, $d_x$ coincides with the degree of the vertex $x$ in the graph, that is,  the number of edges containing the vertex $x$. We will assume that $0 \not= d_x < +\infty$ for every $x \in V(G)$.

  For each $x \in V(G)$  we define the following probability measure
$$ m^G_x:=  \frac{1}{d_x}\sum_{y \sim x} w_{xy}\,\delta_y.
$$
 We have that $[V(G), d_G, (m^G_x)]$ is a metric random walk space. It is not difficult to see that the measure $\nu_G$ defined by
 $$\nu_G(A):= \sum_{x \in A} d_x,  \quad A \subset V(G)$$
is an invariant and  reversible measure for this random walk.

}
 \end{example}

\begin{definition}{\rm
 Given a $\nu$-measurable set $\Omega \subset X$, we define its {\it $m$-boundary} as
  $$\partial_m\Omega:=\{ x\in X\setminus \Omega : m_x(\Omega)>0 \}$$
  and its {\it $m$-closure} as $$\Omega_m:=\Omega\cup\partial_m\Omega.$$}
  \end{definition}

\begin{assumption}\label{as2}{\rm
From now on we assume that  $\Omega\subset X$ is a $\nu$-measurable set satisfying
    $$\nu(\Omega_m) < +\infty.$$}
\end{assumption}

\subsection{Completely accretive operators and semigroup theory}\label{secacc}
Since Semigroup Theory will be used along the paper, we would like to conclude this introduction with some notations and results from this theory along with results from the theory of completely accretive operators (see \cite{BCr2}, \cite{Brezis} and \cite{CrandallLiggett}, or   the Appendix in \cite{ElLibro}, for more details).
We denote by $J_0$ and $P_0$ the following sets of functions:
$$J_0 := \{ j : \R \rightarrow [0, +\infty] \ : \ \mbox{$j$ is convex,
lower semi-continuous and} \ j(0) = 0 \},$$
  $$ P_0:= \left\{q\in  C^\infty(\R) \ : \ 0\le q'\le 1, \hbox{ supp}(q')  \hbox{ is compact and }
  0\notin \hbox{supp}(q) \right\}.
  $$
Assume now that $\nu(X) < \infty$. Let $u,v\in L^1(X,\nu)$. The following relation between $u$ and $v$ is defined in \cite{BCr2}:
$$
u\ll v \ \hbox{ if, and only if,} \ \int_{X} j(u)\,  d\nu  \leq \int_{X} j(v)
 \, d\nu \ \ \hbox{for all} \ j \in J_0.
$$
An operator $\mathcal{A} \subset L^1(X,\nu)\times L^1(X,\nu)$ is called {\it completely accretive} if, for every $(u_i, v_i) \in \mathcal{A}$, $i=1,2$,
  and $\lambda >0$, one
  has that
  \begin{displaymath}
    u_1 - u_2 \ll u_1 - u_2 + \lambda (v_1 - v_2).
  \end{displaymath}

The following characterization of complete accretivity is proved in \cite{BCr2}.
\begin{proposition} An operator $\mathcal{A} \subset L^1(X,\nu)\times L^1(X,\nu)$ is  completely accretive if, for every $(u_i, v_i) \in \mathcal{A}$, $i=1,2$,
$$\int_X (v_1 - v_2) q(u_1 - u_2)d\nu \geq 0 \quad \hbox{for every} \ \ q \in P_0.$$
\end{proposition}

Let $E$ be a linear subspace of $L^1(X,\nu)$.  An operator $\mathcal{A}$ defined in $E$   is m-{\it
completely accretive} in $E$ if $\mathcal{A}$ is completely accretive and $R(I
+ \lambda \mathcal{A}) = E$ for all $\lambda > 0$ (or, equivalently, for some $\lambda>0$).

A Banach space $(E, \Vert \ \Vert_E)$ with $E \subset L^1(X,\nu)$ is a {\it normal Banach space} if it has the following property: $$u \in E, \ v \in L^1(X,\nu), \ \ v \ll u \ \Rightarrow v \in E \ \ \hbox{and} \ \ \Vert v \Vert_E \leq \Vert u\Vert_E.$$
Examples of normal Banach spaces are the spaces $L^p(X,\nu)$, $1 \leq p \leq \infty$.

\begin{theorem}[\cite{BCr2}]\label{teointronls}
   If $\mathcal{A}$
 is an m-completely accretive operator in a normal Banach space  $E \subset L^1(X,\nu)$, then, for every $u_0 \in \overline{D(\mathcal{A})}$, there exists a unique mild solution
  of the problem
\begin{equation}\label{AACCPP}
\left\{
\begin{array}{l}
\displaystyle \frac{du(t)}{dt} + \mathcal{A}u(t) \ni 0, \quad t \in (0, \infty)\\[10pt]
u(0) = u_{0}.
\end{array}
\right.
\end{equation}
Moreover, if $u_0 \in D(\mathcal{A})$, then the mild solution of \eqref{AACCPP} is a strong solution,  that is, the equation in \eqref{AACCPP} is satified for almost all $t \in (0, \infty)$.

Furthermore, we have the following contraction and maximum principle in any $L^q(X,\nu)$ space, $1\le q\le +\infty$: for $u_{1,0},u_{2,0} \in \overline{D(\mathcal{A})}$ and denoting by $u_i$  the  unique mild solution
  of the problem
$$
\left\{
\begin{array}{l}
\displaystyle \frac{du_i(t)}{dt} + \mathcal{A}u_i(t) \ni 0, \quad t \in (0, \infty)\\[10pt]
u_i(0) = u_{i,0},
\end{array}
\right.
$$
$ i=1,2,$ we have
$$\Vert (u_1(t)-u_2(t))^+\Vert_{L^q(X,\nu)}\le \Vert (u_{1,0}-u_{2,0})^+\Vert_{L^q(X,\nu)}\quad \forall\, 0<t<T.
$$
In addition, if $\mathcal{A}$ is positively homogeneous of degree $0 < m \not= 1$, i.e., $\mathcal{A}(\lambda u) = \lambda^m u$ for $u \in D(\mathcal{A})$, then, for every $u_0 \in \overline{D(\mathcal{A})}$, the mild solution of \eqref{AACCPP} is a strong solution.
\end{theorem}

\section{The  nonhomogeneous  Neumann problem for  evolution problems of   Leray-Lions type}

In this section we will give our main results concerning the existence and uniqueness of solutions for the nonhomogeneous  Neumann problem for  evolution problems of   Leray-Lions type in metric random walk spaces. We consider two different types of Neumann boundary conditions. We start with the definition of nonlocal Leray-Lions operator.

\subsection{Nonlocal Leray--Lions operators}
 Recall the definition of the generalized product measure $\nu \otimes m_x$ (see, for instance, \cite[Definition 2.2.7]{Ambrosio}), it is defined as the measure in $X \times X$ given by
 \begin{equation}\label{forarb002} \nu \otimes m_x(U) := \int_X   \int_X \1_{U}(x,y) dm_x(y)   d\nu(x)\quad\hbox{for }  U\in \mathcal{B}(X\times X),
 \end{equation} where it is required that the map $x \mapsto m_x(E)$ is $\nu$-measurable for any Borel set $E \in \mathcal{B}(X)$.
 Moreover, it holds that $$\int_{X \times X} g  d(\nu \otimes m_x)   = \int_X   \int_X g(x,y) dm_x(y)  d\nu(x)$$
for every  $g\in L^1(X\times X,\nu\otimes m_x)$.

For $1<p<+\infty$, let us consider a function $\a_p:X\times X\times \mathbb{R}\to \mathbb{R}$ such that
$$ (x,y)\mapsto \a_p(x,y,r) \quad \hbox{is $\nu\otimes m_x$-measurable for all $r$;}
$$
\begin{equation}\label{ll002} \hbox{$\a_p(x,y,.)$  is continuous for $\nu\otimes m_x$-a.e $(x,y)\in X\times X$;}
\end{equation}
\begin{equation}\label{llo4}
 \a_p(x,y,r)=-\a_p(y,x,-r) \quad \hbox{for $\nu\otimes m_x$-a.e $(x,y)\in X\times X$ and for all $r$;}
\end{equation}
\begin{equation}\label{llo3}
(\a_p(x,y,r)-\a_p(x,y,s))(r-s) > 0 \quad \hbox{for $\nu\otimes m_x$-a.e. $(x,y)$ and for all $r\neq s$;}
\end{equation}
there exist constants $c,C>0$ such that
\begin{equation}\label{llo1}
|\a_p(x,y,r)|\le C\left(1+|r|^{p-1}\right) \quad \hbox{for $\nu\otimes m_x$-a.e. $(x,y)\in X\times X$ and for all $r$,}
\end{equation}
and
\begin{equation}\label{llo2}
\a_p(x,y,r)r\ge c\vert r \vert^p \quad \hbox{for $\nu\otimes m_x$-a.e. $(x,y)\in X\times X$ and for all $r$.}
\end{equation}

This last condition implies that
$$
\a_p(x,y,0)=0 \ \hbox{ and } \   \hbox{sign}_0(\a_p(x,y,r))=\hbox{sign}_0(r) \quad \hbox{for $\nu\otimes m_x$-a.e. $(x,y)\in X\times X$}.
$$

Given a function $u : X \rightarrow \R$ we define its nonlocal gradient $\nabla u: X \times X \rightarrow \R$ as
$$\nabla u (x,y):= u(y) - u(x) \quad \forall \, x,y \in X.$$
For a function $\z : X \times X \rightarrow \R$, its {\it $m$-divergence} ${\rm div}_m \z : X \rightarrow \R$ is defined as
 $$({\rm div}_m \z)(x):= \frac12 \int_{X} (\z(x,y) - \z(y,x)) dm_x(y).$$

An example of  a function $\a_p$ satisfying the above assumptions  is
$$\a_p(x,y,r):=\frac{\varphi(x)+\varphi(y)}{2}|r|^{p-2}r,$$ being  $\varphi:X\rightarrow \R$ a $\nu$-measurable function satisfying $0<c\le \varphi\le C$ where $c$ and $C$ are constants.
In particular, if $\varphi=1$, we have that
$$\begin{array}{l}\displaystyle
\hbox{div}_m\big(\a_p(x,y,u(y)-u(x)\big)(x)= \frac{1}{2}\int_{X}  |u(y)-u(x)|^{p-2}(u(y)-u(x))  dm_x(y)
\end{array}
$$
is the $p$-Laplacian operator on the metric random walk space.

\medskip

Let us point out that,
for the random walk  $m^J$,  Karami, Sadik and  Ziad, in \cite{KSZ}, study a  homogeneous Neumann problem of the type~\eqref{sinper} as a nonlocal model for denoising. They take $\a_p(x,y,r) = \vert r \vert^{p(x,y)-2} r$ with $p(x,y)$ continuous, symmetric and satisfying
$$1 <p^-_1 \leq  \inf_{y \in \Omega} p(x,y) \leq p^-_2 < +\infty, \quad 1 <p^+_1 \leq  \sup_{y \in \Omega} p(x,y) \leq p^+_2 < +\infty, \quad \forall \, x \in \Omega,$$
Furthermore,
Galiano in \cite{Galiano}, studies this type of homogenous Neumann problem for $\a_p(x,y,r)$   bounded in $(x,y)$, continuous in $r$ and satisfying \eqref{llo4} and \eqref{llo3}.

\subsection{Neumann boundary operators}\label{sectcom}

  We define   the   {\it nonlocal Neumann boundary operator} (of  Gunzburger--Lehoucq type) by
$$
\mathcal{N}^{\a_p}_1 u(x):=  -\int_{\Omega_m} \a_p(x,y,u(y)-u(x)) dm_x(y)    \quad \hbox{for } x \in \partial_m\Omega,
$$
and the   {\it nonlocal Neumann boundary  operator} (of Dipierro--Ros-Oton--Valdinoci type) as
$$
\mathcal{N}^{\a_p}_2 u(x):=  -\int_{\Omega} \a_p(x,y,u(y)-u(x)) dm_x(y)    \quad \hbox{for } x \in \partial_m\Omega .
$$

For each of these Neumann boundary operators our main goal is to study the evolution problem
\begin{equation}\label{sabore001}
\left\{ \begin{array}{ll} u_t(t,x) = \hbox{div}_m\a_p u(t,x),    &x\in  \Omega,\ 0<t<T, \\ \\ \mathcal{N}^{\a_p}_\mathbf{j} u(t,x) = \varphi(x),    &x\in\partial_m\Omega, \  0<t<T, \\ \\ u(0,x) = u_0(x),    &x\in \Omega, \end{array} \right.
\end{equation}
$j=1$, $2$, and the following associated
  Neumann  problem
 \begin{equation}\label{Neumann1}
\left\{ \begin{array}{ll} u(x)-\hbox{div}_m\a_p u(x) = \varphi(x), \quad &  x\in\Omega, \\ [8pt]  \mathcal{N}^{\a_p}_\mathbf{j} u(x) = \varphi(x),  \quad & x\in\partial_m\Omega . \end{array} \right.
\end{equation}
In~\eqref{sabore001} and~\eqref{Neumann1} we have used the following simplified notation
 $$\hbox{div}_m \a_p u(t,x):=\hbox{div}_m\big(\a_p(x,y,u(t,y)-u(t,x)\big)(x)$$
and
$$\hbox{div}_m \a_p u(x):=\hbox{div}_m\big(\a_p(x,y,u(y)-u(x)\big)(x).$$

Observe that $\hbox{div}_m\a_p$ is
  {\it a  kind of Leray--Lions operator for the random walk $m$}. On account of~\eqref{llo4}, we have that
$$
\begin{array}{c}
\displaystyle\hbox{div}_m\a_p u (x) =\frac12 \int_{X}\big(\a_p(x,y,u(y)-u(x)) - \a_p(y,x,u(x)-u(y))\big) dm_x(y)
\\ \\
\displaystyle
=\int_X \a_p(x,y,u(y)-u(x)) dm_x(y).
\end{array}
$$
Moreover, by the reversibility of $\nu$ with respect to $m$, we have that $m_x(X\setminus \Omega_m)=0$ for $\nu$-a.e. $x\in \Omega$. Indeed,
$$\displaystyle \int_{\Omega}m_x(X\setminus \Omega_m)d\nu(x)=\int_{X\setminus \Omega_m}m_x(\Omega)d\nu(x)=0  . $$
Consequently,
$$
 \hbox{div}_m\a_p u (x) =\int_{\Omega_m} \a_p(x,y,u(y)-u(x)) dm_x(y) \quad  \hbox{for every $x\in\Omega$.}
$$

The following integration by parts formula follows by the reversibility of $\nu$ with respect to~$m$.

  \begin{lemma}\label{intpart}  Let $q\ge 1$. If $Q \subset X \times X$ is a symmetric set (i.e., $(x,y) \in Q \iff (y,x) \in Q$) and $\Psi : Q \rightarrow \R$  is a  $\nu\otimes m_x$-a.e. antisymmetric function (i.e., $\Psi(x,y) = - \Psi(y,x)$  for $\nu\otimes m_x$-a.e. $(x,y)\in Q$) with  $\Psi \in L^{q}(Q,  \nu\otimes m_x )$ and $u\in L^{q'}(X,\nu)$ then
  $$\int_Q \Psi(x,y) u(x) d(\nu\otimes m_x)(x,y) = - \frac{1}{2} \int_Q \Psi(x,y)(u(y) - u(x)) d(\nu\otimes m_x)(x,y).$$
In particular, if $\Psi \in L^1(Q, \nu\otimes m_x )$,
$$\int_Q \Psi(x,y) d(\nu\otimes m_x)(x,y) = 0.$$

 \end{lemma}

Consequently, we obtain the following integration by parts formula.
  Let $$Q_1=\Omega_m\times\Omega_m$$
  and
$$Q_2=(\Omega_m\times\Omega_m)\setminus(\partial_m\Omega\times\partial_m\Omega).$$

  \begin{proposition}\label{prop23} Let $\mathbf{j}\in\{1,2\}$. Let $u$ be a $\nu$-measurable function such that
   $$(x,y)\mapsto \a_p(x,y,u(y)-u(x))\in L^{q}(Q_\mathbf{j},\nu\otimes m_x)$$ and let $w \in L^{q'}(\Omega_m)$, then
$$ \begin{array}{l}
 \displaystyle
-\int_\Omega  \hbox{\rm div}_m\a_p u (x)w(x)d\nu(x) +\int_{\partial_m\Omega}  \mathcal{N}^{\a_p}_\mathbf{j} u(x) w(x)d\nu(x) \\ \\
\qquad\qquad\qquad= \displaystyle\frac{1}{2} \int_{Q_\mathbf{j}} \a_p(x,y,u(y)-u(x)) (w(y) - w(x)) d(\nu\otimes m_x)(x,y) .
\end{array}
$$
 \end{proposition}
\begin{proof} Applying Lemma \ref{intpart}, with $\Psi(x,y)  = \a_p(x,y,u(y)-u(x))$ and with $Q= Q_\mathbf{j}$, we have
$$\begin{array}{l}\displaystyle\frac{1}{2}\int_{Q_\mathbf{j}} \a_p(x,y,u(y)-u(x)) (w(y) - w(x)) d(\nu\otimes m_x)(x,y)
\\ \\ \displaystyle
\quad= - \int_{Q_\mathbf{j}} \a_p(x,y,u(y)-u(x)) w(x) d(\nu\otimes m_x)(x,y)
 \\ \\ \displaystyle
\quad = -\int_\Omega   \hbox{div}_m\a_p u (x)w(x)d\nu(x) +  \int_{\partial_m\Omega}   \mathcal{N}^{\a_p}_\mathbf{j} u(x)w(x) d\nu(x).
 \end{array}
 $$
\end{proof}

As a corollary, since $\nu(\Omega_m) < +\infty$, we have the following nonlocal form of the divergence theorem.

  \begin{proposition} Let $\mathbf{j}\in\{1,2\}$. If $u \in L^p(\Omega_m,\nu)$, then
$$
 \int_\Omega \hbox{div}_m\a_p u (x)d\nu(x) =   \int_{\partial_m\Omega} \mathcal{N}^{\a_p}_\mathbf{j} u(x) d\nu(x).
$$
 \end{proposition}

\begin{remark}\label{remmon}
Let us see, formally, the way in which we will be using Proposition~\ref{prop23} in what follows. Suppose that we are in the following situation:
$$
\left\{ \begin{array}{ll} -\hbox{div}_m\a_p u(x) = f(x), \quad & x\in\Omega, \\ [8pt] \mathcal{N}^{\a_p}_\mathbf{j} u(x) = g(x), \quad & x \in\partial_m\Omega, \end{array} \right.
$$
for $\mathbf{j}=1$ or $2$. Then, multiplying the first equation by a function $w$, defined in $\Omega_m$ and with adequate integrability, integrating over $\Omega$ and using the integration by parts formula, we get
$$ \begin{array}{l}
 \displaystyle
  \frac{1}{2} \int_{Q_\mathbf{j}} \a_p(x,y,u(y)-u(x)) (w(y) - w(x)) d(\nu\otimes m_x)(x,y) \\ \\ \qquad \displaystyle = \int_\Omega f(x)w(x)d\nu(x) +\int_{\partial_m\Omega} g(x)w(x) d\nu(x) .
\end{array}
$$
Moreover, as a consequence of these computations, if
$$
\left\{ \begin{array}{ll} -\hbox{div}_m\a_p u_i(x) = f_i(x), \quad & x\in\Omega, \\ [8pt] \mathcal{N}^{\a_p}_\mathbf{j} u_i(x) = g_i(x), \quad & x \in\partial_m\Omega, \end{array} \right.
$$
 $i=1,2$, then, for a nondecreasing function $T:\mathbb{R}\to \mathbb{R}$, we obtain
$$ \begin{array}{l}
 \displaystyle
 \int_\Omega (f_1(x)-f_2(x))T(u_1(x)-u_2(x))d\nu(x) +\int_{\partial_m\Omega} (g_1(x)-g_2(x))T(u_1(x)-u_2(x)) d\nu(x)  \\ \\
 \displaystyle
 \qquad=\frac{1}{2} \int_{Q_\mathbf{j}} \big(\a_p(x,y,u_1(y)-u_1(x))-\a_p(x,y,u_2(y)-u_2(x))\big) \times\qquad\qquad
  \\[14pt] \displaystyle
  \hfill\qquad\qquad \times\big(T(u_1(y) - u_2(y)) -T(u_1(x)-u_2(x))\big)d(\nu\otimes m_x)(x,y)\,\ge 0 .
\end{array}
$$

\noindent  Indeed, for $\nu\otimes m_x$--a.e. $(x,y)\in Q_\mathbf{j}$ ( $\mathbf{j}=1$ or $2$) satisfying $u_1(y)-u_1(x)\geq u_2(y)-u_2(x)$, \ by \eqref{llo3}, we have that
$$\a_p(x,y,u_1(y)-u_1(x))-\a_p(x,y,u_2(y)-u_2(x))\geq 0.$$
Moreover, for these $(x,y)$, since $T$ is non-decreasing and $u_1(y) - u_2(y)\geq u_1(x)-u_2(x)$,
$$T(u_1(y) - u_2(y)) -T(u_1(x)-u_2(x))\geq 0 .$$
Similarly, for $\nu\otimes m_x$--a.e. $(x,y)$ satisfying $u_1(y)-u_1(x)\leq u_2(y)-u_2(x)$, we get
$$\big(\a_p(x,y,u_1(y)-u_1(x))-\a_p(x,y,u_2(y)-u_2(x))\big)\times \big(T(u_1(y) - u_2(y)) -T(u_1(x)-u_2(x))\big)\geq 0 . $$
  \end{remark}

\medskip

\subsection{Neumann boundary conditions of   Gunzburger--Lehoucq type}\label{aus23}

In this subsection we study the  problem
\begin{equation}\label{Neumann4}
\left\{ \begin{array}{ll} u_t(t,x) = \hbox{div}_m\a_p u(t,x),    &x\in  \Omega,\ 0<t<T, \\ \\ \mathcal{N}^{\a_p}_1 u(t,x) = \varphi(x),    &x\in\partial_m\Omega, \  0<t<T, \\ \\ u(0,x) = u_0(x),    &x\in \Omega. \end{array} \right.
\end{equation}
In addition to Assumptions~\ref{as1} and \ref{as2} we will also work under the following assumption.
\begin{assumption}\label{as3} {\rm We assume that \rm $[\Omega_m,d,m,\nu]$ satisfies the following {\it Poincar\'{e}  type inequality}: there exists a constant $\lambda>0$ such that, for any $u \in L^p(\Omega_m,\nu)$,
\begin{equation}\label{agost004} \left\Vert  u \right\Vert_{L^p(\Omega_m,\nu)}  \leq \lambda\left(\left(\int_{Q_1} |u(y)-u(x)|^p d(\nu\otimes m_x)(x,y) \right)^{\frac1p}+\left| \int_{\Omega } u\,d\nu\right|\right).
\end{equation}
or, equivalently,
$$ \left\Vert  u - \frac{1}{\nu(\Omega)} \int_\Omega u d\nu \right\Vert_{L^p(\Omega_m,\nu)}  \leq \lambda\left(\int_{Q_1} |u(y)-u(x)|^p d(\nu\otimes m_x)(x,y) \right)^{\frac1p}.
$$
}
\end{assumption}

It is shown in~\cite{MST2} (see also~\cite{BLGneu} and~\cite{ElLibro}) that, under rather general conditions, there are metric random walk spaces satisfying this kind of inequality. Note that the proof of the existence of the Poincar\'e type inequality in~\cite{MST2} must be slightly modified in order to cover the inequality considered in \eqref{agost004}.

 To study~\eqref{Neumann4} we will use the Nonlinear Semigroup Theory, to this end we define the  following operator in   $ L^1(\Omega,\nu)\times L^1(\Omega,\nu)$ associated to the problem. Observe that the space of definition is $L^1(\Omega,\nu)$  and not $ L^1(\Omega_m,\nu)$.

\begin{definition} {\rm  Let $\varphi\in L^1(\partial_m\Omega,\nu)$. We say that $(u,v) \in B^m_{\a_p,\varphi}$ if $ u,v \in L^1(\Omega,\nu)$
  and there exists  $ \overline u\in L^p(\Omega_m,\nu)$ (that we will denote equally as $u$) such that $\overline{u}_{\vert \Omega} = u$,
$$
(x,y)\mapsto a_p(x,y,u(y)-u(x))\in L^{p'}(Q_1,\nu\otimes m_x)
$$
and
$$
\left\{ \begin{array}{ll} -\hbox{div}_m\a_p u = v \quad &\hbox{in} \ \ \Omega, \\ [8pt] \mathcal{N}^{\a_p}_1 u = \varphi \quad &\hbox{in} \ \ \partial_m\Omega; \end{array} \right.
$$
that is,
$$
v(x) = -  \int_{\Omega_m} \a_p(x,y,u(y)-u(x)) dm_x(y), \quad x \in \Omega,
$$
and
$$
\varphi(x) = -\int_{\Omega_m} \a_p(x,y,u(y)-u(x)) dm_x(y), \quad x \in \partial_m\Omega.
$$
}
\end{definition}

\begin{remark}\label{porsi001}{
Observe that if $(u,v) \in B^m_{\a_p,\varphi}$  then $v\in L^{p'}(\Omega,\nu)$ and, moreover,
$$\int_\Omega v(x)d\nu(x)+\int_{\partial_m\Omega}\varphi(x) d\nu(x)=0.$$}
\end{remark}

  \begin{theorem}\label{teo27} Let $\varphi\in L^{p'}(\partial_m\Omega,\nu)$. The  operator $B^m_{\a_p,\varphi}$ is  completely accretive and satisfies the range condition
\begin{equation}\label{rangecond01}L^{p'}(\Omega,\nu)\subset R(I+ B^m_{\a_p,\varphi}).
\end{equation}
Consequently,   $B^m_{\a_p,\varphi}$ is $m$-completely accretive in $L^{p'}(\Omega,\nu)$.

\end{theorem}

\begin{proof}
To prove the complete accretivity of the operator $B^m_{\a_p,\varphi}$ we need to show that, if $(u_i, v_i) \in B^m_{p,\varphi}$, $i=1,2$, and  $q \in P_0$, then
$$
\int_\Omega (v_1(x) - v_2(x)) q(u_1(x) - u_2(x)) d\nu(x) \geq 0.
$$
In fact, by  the integration by parts formula given in  Proposition~\ref{prop23} and
having in mind that, for both $i=1$ and $2$,
$$\varphi(x) = -\int_{\Omega_m}  \a_p(x,y,u_i(y)-u_i(x)) dm_x(y), \quad x \in \partial_m\Omega, $$ we get (see also Remark~\ref{remmon})
$$\begin{array}{l}
\displaystyle
\int_\Omega (v_1(x) - v_2(x)) q(u_1(x) - u_2(x)) dx
 \\ \\
 \displaystyle\qquad =\frac12\int_{Q_1} \big( \a_p(x,y,u_1(y)-u_1(x))  - \a_p(x,y,u_2(y)-u_2(x)) \big)\times\qquad\qquad\\ \\
 \hfill \displaystyle \qquad\qquad\times\big(q(u_1(y) - u_2(y))-q(u_1(x) - u_2(x))\big) d(\nu\otimes m_x)(x,y)\ge 0\, .
  \end{array}
 $$

  Let us see that $B^m_{\a_p,\varphi}$ satisfies the range condition~\eqref{rangecond01}; that is, let us prove that for $\phi\in L^{p'}(\Omega,\nu)$ there exists $u\in D(B^m_{\a_p,\varphi})$ such  that $$u+B^m_{\a_p,\varphi}u\ni\phi.$$
  Take  the following $L^\infty$ approximations of $\phi$ and $\varphi$: $\phi_{n,k} := \sup\{
\inf\{\phi,n\},-k\}$ and $\varphi_{n,k} := \sup\{
\inf\{\varphi,n\},-k\}$, which are nondecreasing in $n$ and nonincreasing in $k$.
 Following the  idea used in~\cite{AIMTq} and~\cite{BLGneu},
 for $n, k\in \mathbb{N}$ and $K>0$,
let $$A:L^p(\Omega_m,\nu) \rightarrow  L^{p'}(\Omega_m,\nu)\equiv L^{p'}(\Omega,\nu)\times L^{p'}(\partial_m\Omega,\nu)$$ be defined by
$$A(u)= \big(A_1(u),A_2(u)\big),$$ where
$$A_1(u)(x)=T_K(u)(x)-\int_{\Omega_m}\a_p(x,y,u(y)-u(x))dm_x(y)+
\frac{1}{n}|u(x)|^{p-2}u^+(x)
-\frac{1}{k}|u(x)|^{p-2}u^-(x),$$ for $x\in\Omega$,
and
$$A_2(u)(x)= -\int_{\Omega_m}\a_p(x,y,u(y)-u(x))dm_x(y)
+\frac{1}{n}|u(x)|^{p-2}u^+(x)
-\frac{1}{k}|u(x)|^{p-2}u^-(x),$$
for $x\in \partial_m\Omega$. Here, $T_K$ is the truncation operator defined as $$T_K(r):= \left\{ \begin{array}{lll}r \quad &\hbox{if} \ \vert r \vert \leq K, \\[4pt]K \quad &\hbox{if} \ r >K, \\[4pt]-K \quad &\hbox{if} \ r< -K.  \end{array}\right.$$

  It is easy to see that $A$ is continuous and, moreover, it is monotone and coercive in $L^p(\Omega_m, \nu)$.
Indeed, for the monotonicity, follow the same steps used in the first part of this proof
and, for the coercivity, observe that
$$\int_{\Omega_m} A(u)u d\nu\ge\frac{1}{n}\Vert u^+\Vert_{L^p(\Omega_m,\nu)}^p+\frac{1}{k}\Vert u^-\Vert_{L^p(\Omega_m,\nu)}^p .$$

Therefore,  since $(\phi_{n,k},\varphi_{n,k})\in L^{p'}(\Omega)\times L^{p'}(\partial_m\Omega)$, by  \cite[Corollary 30]{brezisgrenoble},  there exist $u_{n,k}\in L^{p}(\Omega_m, \nu)$ such that
$$\big(A_1(u_{n,k}),A_2(u_{n,k})\big)=(\phi_{n,k},\varphi_{n,k}).$$
That is,
\begin{equation}\label{E1} \begin{array}{l}T_K(u_{n,k})(x)-\displaystyle\int_{\Omega_m}\a_p(x,y,u_{n,k}(y)-u_{n,k}(x))dm_x(y)+
\frac{1}{n}|u_{n,k}(x)|^{p-2}u_{n,k}^+(x)
\\  \qquad -\displaystyle \frac{1}{k}|u_{n,k}(x)|^{p-2}u_{n,k}^-(x)=\phi_{n,k}(x), \  \hbox{ for every $x\in\Omega,$}\end{array}\end{equation}
and
\begin{equation}\label{E2} \begin{array}{l} \displaystyle -\int_{\Omega_m}\a_p(x,y,u_{n,k}(y)-u_{n,k}(x))dm_x(y)
+\frac{1}{n}|u_{n,k}(x)|^{p-2}u_{n,k}^+(x)
\\  \displaystyle \ \qquad -\frac{1}{k}|u_{n,k}(x)|^{p-2}u_{n,k}^-(x)=\varphi_{n,k}(x), \ \hbox{ for every $x\in \partial_m\Omega.$}\end{array} \end{equation}

We will now see that $u_{n,k}\in L^\infty(\Omega,\nu)$, $n$, $k\in\N$. Let
$$M=M_{\phi,\varphi,n,k}:=\max\left\{\Vert \phi_{n,k}\Vert_{L^\infty(\Omega,\nu)},\left(n\Vert \varphi_{n,k}\Vert_{L^\infty(\partial_m\Omega,\nu)}\right)^\frac{1}{p-1},\left(k\Vert \varphi_{n,k}\Vert_{L^\infty(\partial_m\Omega,\nu)}\right)^\frac{1}{p-1}\right\}.$$
 Then, multiplying \eqref{E1} by $(u_{n,k}-M)^+$ and integrating over $\Omega$, since $$|u_{n,k}(x)|^{p-2}u_{n,k}^-(x)(u_{n,k}(x)-M)^+=0 \quad \hbox{ for every $x\in\Omega_m$},$$
 we get
$$\int_\Omega T_K(u_{n,k})(x) (u_{n,k}(x)-M)^+ d\nu(x) \qquad \qquad  \qquad\qquad \qquad \qquad \qquad  \qquad  \qquad \qquad \qquad $$
$$  -\displaystyle\int_\Omega\int_{\Omega_m}\a_p(x,y,u_{n,k}(y)-u_{n,k}(x))(u_{n,k}(x)-M)^+ dm_x(y) d\nu(x)  \qquad \qquad \qquad \qquad$$ $$+
\int_\Omega \frac{1}{n}|u_{n,k}(x)|^{p-2}u_{n,k}^+(x)(u_{n,k}(x)-M)^+ d\nu(x) = \int_\Omega \phi_{n,k}(x)(u_{n,k}(x)-M)^+ d \nu(x).$$
On the other hand, by Proposition~\ref{prop23} with $w=(u_{n,k}-M)^+$ and having in mind \eqref{E2},
we get that
$$-\int_{\Omega}\int_{\Omega_m}\a_p(x,y,u_{n,k}(y)-u_{n,k}(x))(u_{n,k}(x)-M)^+ dm_x(y)d\nu(x)$$
$$= \frac12 \iint_{Q_1}\a_p(x,y,u_{n,k}(y)-u_{n,k}(x))\left((u_{n,k}(y)-M)^+-(u_{n,k}(x)-M)^+\right) dm_x(y)d\nu(x)$$
$$+\int_{\partial_m\Omega}\left(\int_{\Omega_m}\a_p(x,y,u_{n,k}(y)-u_{n,k}(x))dm_x(y)\right)(u_{n,k}(x)-M)^+d\nu(x)$$
$$\ge \int_{\partial_m\Omega}\left(\int_{\Omega_m}\a_p(x,y,u_{n,k}(y)-u_{n,k}(x))dm_x(y)\right)(u_{n,k}(x)-M)^+d\nu(x)$$ $$= \int_{\partial_m\Omega}\left(-\varphi_{n,k}(x)+\frac{1}{n}|u_{n,k}(x)|^{p-2}u_{n,k}^+(x)
-\frac{1}{k}|u_{n,k}(x)|^{p-2}u_{n,k}^-(x)\right)(u_{n,k}(x)-M)^+d\nu(x) $$
$$= \int_{\partial_m\Omega}\left(-\varphi_{n,k}(x)+\frac{1}{n}|u_{n,k}(x)|^{p-2}u_{n,k}^+(x)
\right)(u_{n,k}(x)-M)^+d\nu(x) $$
$$ = -\int_{\partial_m\Omega}\varphi_{n,k}(x)(u_{n,k}(x)-M)^+d\nu(x) + \frac{1}{n}\int_{\partial_m\Omega}|u_{n,k}(x)|^{p-2}u_{n,k}^+(x)
(u_{n,k}(x)-M)^+d\nu(x) .$$
Hence,
$$ \int_\Omega T_K(u_{n,k})(x)(u_{n,k}(x)-M)^+d\nu(x)+\frac{1}{n}\int_{\partial_m\Omega}|u_{n,k}(x)|^{p-2}u_{n,k}^+(x)
(u_{n,k}(x)-M)^+d\nu(x)$$
$$\leq \int_\Omega \phi_{n,k}(x)(u_{n,k}(x)-M)^+ d \nu(x)+\int_{\partial_m\Omega}\varphi_{n,k}(x)(u_{n,k}(x)-M)^+d\nu(x).$$
 Therefore,
$$\begin{array}{c}
\displaystyle\int_{\Omega}\big(T_K(u_{n,k})-M\big)(u_{n,k}-M)^+ d\nu + \frac{1}{n}\int_{\partial_m\Omega}\big(|u_{n,k}|^{p-2}u_{n,k}^+-M^{p-1})(u_{n,k}-M)^+ d\nu
\\ \\ \displaystyle
\le
\int_{\Omega}(\phi_{n,k}-M) (u_{n,k}-M)^+ d\nu + \int_{\partial_m\Omega}(\varphi_{n,k}-\frac1n M^{p-1}) (u_{n,k}-M)^+ d\nu\le 0,
\end{array}$$
 and, consequently,   taking $K>M$, we get
$$u_{n,k}\le M\quad\nu-\hbox{a.e. in }\hbox{$\Omega$}.$$
Similarly,   taking $w=(u_{n,k}+M)^-$, we get
  $$\begin{array}{c}
\displaystyle\int_{\Omega}\big(T_K(u_{n,k})+M\big)(u_{n,k}+M)^-d\nu - \frac{1}{k}\int_{\partial_m\Omega}\big(|u_{n,k}|^{p-2}u_{n,k}^- -M^{p-1})(u_{n,k}+M)^-d\nu
\\ \\ \displaystyle
\ge
\int_{\Omega}(\phi_{n,k}+M) (u_{n,k}+M)^-d\nu + \int_{\partial_m\Omega}(\varphi_{n,k}+\frac1k M^{p-1}) (u_{n,k}+M)^-d\nu\ge 0,
\end{array}$$
which yields,   taking $K>M$,
$$u_{n,k}\ge -M\quad\nu-\hbox{a.e. in }\hbox{$\Omega$},$$
thus $$\Vert u_{n,k}\Vert_{ L^\infty(\Omega,\nu)}\le M $$
as desired.

Therefore, by \eqref{E1} and \eqref{E2}, we have
\begin{equation}\label{sab1050af}\begin{array}{c}\displaystyle u_{n,k}(x)-\int_{\Omega_m}\a_p(x,y, u_{n,k}(y)-u_{n,k}(x))dm_x(y)+ \qquad\qquad \qquad\qquad\\ \\
\displaystyle \qquad\qquad \qquad\qquad+
\frac{1}{n}|u_{n,k}(x)|^{p-2}u_{n,k}^+(x)
-\frac{1}{k}|u_{n,k}(x)|^{p-2}u_{n,k}^-(x)=\phi_{n,k}(x),
\end{array}\end{equation}
for $x\in\Omega$,
and
\begin{equation}\label{sab1050f}\begin{array}{c}\displaystyle
-\int_{\Omega_m}\a_p(x,y, u_{n,k}(y)-u_{n,k}(x))dm_x(y)+
\qquad\qquad \qquad\qquad\\ \\
\displaystyle \qquad\qquad \qquad\qquad
+\frac{1}{n}|u_{n,k}(x)|^{p-2}u_{n,k}^+(x)
-\frac{1}{k}|u_{n,k}(x)|^{p-2}u_{n,k}^-(x)=\varphi_{n,k}(x),
\end{array}\end{equation}
for $x\in \partial_m\Omega$.

Observe that, by Proposition~\ref{prop23} with $w=T_r(u_{n,k})$ ($r>0$), we  get $$\frac{1}{r}\int_\Omega  u_{n,k}T_r(u_{n,k})d\nu\le \int_\Omega|\phi_{n,k}| d\nu+ \int_{\partial_m\Omega}|\varphi_{n,k}| d\nu
\le \int_\Omega|\phi | d\nu+ \int_{\partial_m\Omega}|\varphi | d\nu
,$$
thus, letting $r\to 0$, we obtain
\begin{equation}\label{agost005}
\int_\Omega  |u_{n,k}|d\nu\le \int_\Omega|\phi| d\nu+\int_{\partial_m\Omega}|\varphi| d\nu.
\end{equation}

  Now, let us see that $\{u_{n,k}\} $ is  $\nu$-a.e. nondecreasing in $n$ and nonincreasing in $k$.  Indeed, for $n'<n$,   by  Proposition~\ref{prop23} with $w(x)=\left(u_{n',k}(x)-u_{n,k}(x)\right)^+$,  $x\in\Omega_m$,
we get
$$
 \begin{array}{l}
\displaystyle
0\le \int_{\Omega_m}\left(\frac{1}{n'}|u_{n',k}(x)|^{p-2}u_{n',k}^+(x)-\frac{1}{n}|u_{n,k}(x)|^{p-2}u_{n,k}^+(x)\right)\left(u_{n',k}(x)-u_{n,k}(x)\right)^+d\nu(x)\\ \\
\displaystyle \qquad +\frac{1}{k}\int_{\Omega_m}-\left(|u_{n',k}(x)|^{p-2}u_{n',k}^-(x)-|u_{n,k}(x)|^{p-2}u_{n,k}^-(x)\right)\left(u_{n',k}(x)-u_{n,k}(x)\right)^+d\nu(x)
\\ \\ \displaystyle
=\int_{\Omega}\left(\phi_{n',k}(x)-\phi_{n,k}(x)\right)\left(u_{n',k}(x)-u_{n,k}(x)\right)^+d\nu(x)
\\ \\ \displaystyle \qquad
+\int_{\Omega}\left(\varphi_{n',k}(x)-\varphi_{n,k}(x)\right)\left(u_{n',k}(x)-u_{n,k}(x)\right)^+d\nu(x)
\le 0 \,. \qquad
\end{array}
$$
Therefore,
\begin{equation}\label{agost006}
  \int_{\Omega_m}\left(\frac{1}{n'}|u_{n',k}(x)|^{p-2}u_{n',k}^+(x)-\frac{1}{n}|u_{n,k}(x)|^{p-2}u_{n,k}^+(x)\right)\left(u_{n',k}(x)-u_{n,k}(x)\right)^+d\nu(x)\\ =0, \end{equation}
  and
\begin{equation}\label{agost007}
\frac{1}{k}\int_{\Omega_m}-\left(|u_{n',k}(x)|^{p-2}u_{n',k}^-(x)-|u_{n,k}(x)|^{p-2}u_{n,k}^-(x)\right)\left(u_{n',k}(x)-u_{n,k}(x)\right)^+d\nu(x)
=0.\end{equation}
 Suppose that $\nu\left(\{x\in\Omega_m : \left(u_{n',k}(x)-u_{n,k}(x)\right)^+\}\right)>0$, i.e., $$\nu\left(\{x\in\Omega_m : u_{n',k}(x)>u_{n,k}(x)\}\right)>0,$$ then, if $\nu\left(\{x\in\Omega_m : u_{n,k}(x)>0\}\right)>0$ or $\nu\left(\{x\in\Omega_m : u_{n,k}(x)<0\le u_{n',k}(x)\}\right)>0$ we get a contradiction with~\eqref{agost006} and, if  $\nu\left(\{x\in\Omega_m : 0> u_{n',k}(x)>u_{n,k}(x)\}\right)>0$, we get a contradiction with~\eqref{agost007}. Consequently, $u_{n,k}$ is $\nu$-a.e. nondecreasing in $n$.

Similarly, we obtain that $u_{n,k}$ is $\nu$-a.e. nonincreasing in $k$.

  Now, multiplying \eqref{sab1050af} and \eqref{sab1050f} by $\displaystyle w= u_{n,k}-\frac{1}{\nu(\Omega)}\int_\Omega u_{n,k}d\nu$, by  Proposition \ref{prop23} (see also Remark ~\ref{remmon}), we get
$$ \begin{array}{l}
 \displaystyle
 \left\Vert u_{n,k}-\frac{1}{\nu(\Omega)}\int_\Omega u_{n,k}d\nu\right\Vert_{L^2(\Omega,\nu)}^2
 \\ \\
 \displaystyle
 \  +\frac{1}{2} \int_{Q_1} \a_p(x,y,u_{n,k}(y)-u_{n,k}(x)) (u_{n,k}(y) - u_{n,k}(x)) d(\nu\otimes m_x)(x,y)
 \\ \\
 \displaystyle = \int_\Omega \phi_{n,k} \left(u_{n,k}-\frac{1}{\nu(\Omega)}\int_\Omega u_{n,k}d\nu\right) d\nu  +\int_{\partial_m\Omega} \varphi_{n,k} \left(u_{n,k}-\frac{1}{\nu(\Omega)}\int_\Omega u_{n,k}d\nu\right)  d\nu \\ \\
 \displaystyle \   + \int_\Omega \left(\frac{1}{k}|u_{n,k}(x)|^{p-2}u_{n,k}^-(x) - \frac{1}{n}|u_{n,k}(x)|^{p-2}u_{n,k}^+(x)\right)\left(u_{n,k}(x)-\frac{1}{\nu(\Omega)}\int_\Omega u_{n,k}d\nu\right) d\nu(x) \\ \\
 \displaystyle \  + \int_{\partial_m \Omega} \left(\frac{1}{k}|u_{n,k}(x)|^{p-2}u_{n,k}^-(x) - \frac{1}{n}|u_{n,k}(x)|^{p-2}u_{n,k}^+(x)\right)\left(u_{n,k}(x)-\frac{1}{\nu(\Omega)}\int_\Omega u_{n,k}d\nu\right) d\nu(x)\\ \\
\end{array}
$$
For the third summand on the right hand side, since $  F_{n,k} (r):=\frac1k |r|^{p-2}r^--\frac1n|r|^{p-2}r^+$ is nonincreasing, we have that
$$\begin{array}{l}\displaystyle\int_\Omega \left(\frac{1}{k}|u_{n,k}(x)|^{p-2}u_{n,k}^-(x) - \frac{1}{n}|u_{n,k}(x)|^{p-2}u_{n,k}^+(x)\right)\left(u_{n,k}(x)-\frac{1}{\nu(\Omega)}\int_\Omega u_{n,k}d\nu\right) d\nu(x)\\ \\
\displaystyle = \int_\Omega F_{n,k}(u_{n,k}(x))\left(u_{n,k}(x)-\frac{1}{\nu(\Omega)}\int_\Omega u_{n,k}d\nu\right) d\nu(x)\\ \\
\displaystyle \le \int_\Omega F_{n,k}\left(\frac{1}{\nu(\Omega)}\int_\Omega u_{n,k}d\nu\right)\left(u_{n,k}(x)-\frac{1}{\nu(\Omega)}\int_\Omega u_{n,k}d\nu\right) d\nu(x)=0\\ \\
\end{array}$$
and for the fourth summand on the right hand side, using again the monotonicity of $F_{n,k}$ and then~\eqref{agost005}, we get
$$\begin{array}{l}\displaystyle\int_{\partial_m\Omega} \left(\frac{1}{k}|u_{n,k}(x)|^{p-2}u_{n,k}^-(x) - \frac{1}{n}|u_{n,k}(x)|^{p-2}u_{n,k}^+(x)\right)\left(u_{n,k}(x)-\frac{1}{\nu(\Omega)}\int_\Omega u_{n,k}d\nu\right) d\nu(x)\\ \\
\displaystyle = \int_{\partial_m\Omega} F_{n,k}(u_{n,k}(x))\left(u_{n,k}(x)-\frac{1}{\nu(\Omega)}\int_\Omega u_{n,k}d\nu\right) d\nu(x)
\\ \\
\displaystyle \le \int_{\partial_m\Omega} F_{n,k}\left(\frac{1}{\nu(\Omega)}\int_\Omega u_{n,k}d\nu\right)\left(u_{n,k}(x)-\frac{1}{\nu(\Omega)}\int_\Omega u_{n,k}d\nu\right) d\nu(x)
\\ \\
\displaystyle \le \int_{\partial_m\Omega} \left(\frac1k+\frac1n\right)\left|\frac{1}{\nu(\Omega)}\int_\Omega u_{n,k}d\nu\right|^{p-1}\left|u_{n,k}(x)-\frac{1}{\nu(\Omega)}\int_\Omega u_{n,k}d\nu\right| d\nu(x)
\end{array}$$
$$\begin{array}{l}
\displaystyle \le \left(\frac1k+\frac1n\right)\left|\frac{1}{\nu(\Omega)}\int_\Omega u_{n,k}d\nu\right|^{p-1}\nu(\partial_m\Omega)^\frac{1}{p'} \left\Vert u_{n,k}(x)-\frac{1}{\nu(\Omega)}\int_\Omega u_{n,k}d\nu\right\Vert_{L^p(\Omega_m,\nu)}
\\
\displaystyle \le 2\frac{\nu(\partial_m\Omega)^\frac{1}{p'}}{\nu(\Omega)^{p-1}}\left( \Vert\phi\Vert_{L^{1}(\Omega,\nu)}  +\Vert\varphi\Vert_{L^{1}(\partial_m\Omega,\nu)} \right)^{p-1}   \left\Vert u_{n,k}(x)-\frac{1}{\nu(\Omega)}\int_\Omega u_{n,k}d\nu\right\Vert_{L^p(\Omega_m,\nu)}.
\end{array}$$
Let $\alpha:=2\frac{\nu(\partial_m\Omega)^\frac{1}{p'}}{\nu(\Omega)^{p-1}}\left( \Vert\phi\Vert_{L^{1}(\Omega,\nu)}  +\Vert\varphi\Vert_{L^{1}(\partial_m\Omega,\nu)} \right)^{p-1} $. Consequently, by~\eqref{llo2}, we have that
$$ \begin{array}{l}
 \displaystyle
\left\Vert u_{n,k}-\frac{1}{\nu(\Omega)}\int_\Omega u_{n,k}d\nu\right\Vert_{L^2(\Omega,\nu)}^2+\frac{c}{2} \int_{Q_1} |u_{n,k}(y) - u_{n,k}(x)|^p d(\nu\otimes m_x)(x,y)
 \\ \\\displaystyle \le    \left(\Vert\phi\Vert_{L^{p'}(\Omega,\nu)}  +\Vert\varphi\Vert_{L^{p'}(\partial_m\Omega,\nu)} + \alpha \right) \left\Vert u_{n,k}-\frac{1}{\nu(\Omega)}\int_\Omega u_{n,k}d\nu\right\Vert_{L^p(\Omega_m,\nu)}.
\end{array}
$$
Now, by Poincar\'{e}'s inequality~\eqref{agost004},
$$ \begin{array}{l}
 \displaystyle
\left\Vert u_{n,k}-\frac{1}{\nu(\Omega)}\int_\Omega u_{n,k}d\nu\right\Vert_{L^2(\Omega,\nu)}^2+\frac{c}{2\lambda^p}\left\Vert u_{n,k}-\frac{1}{\nu(\Omega)}\int_\Omega u_{n,k}d\nu\right\Vert_{L^p(\Omega_m,\nu)}^p
 \\ \\\displaystyle \le    \left(\Vert\phi\Vert_{L^{p'}(\Omega,\nu)}  +\Vert\varphi\Vert_{L^{p'}(\partial_m\Omega,\nu)}+ \alpha\right) \left\Vert u_{n,k}-\frac{1}{\nu(\Omega)}\int_\Omega u_{n,k}d\nu\right\Vert_{L^p(\Omega_m,\nu)}
 .
\end{array}
$$
 Hence, by~\eqref{agost005}, we have that $\{u_{n,k}\}$ is bounded in $L^p(\Omega_m,\nu)$  (and in $L^2(\Omega,\nu)$).
Then, by the monotone and dominated convergence theorems,  we can pass to the limit in $n$, and then in $k$, in~\eqref{sab1050af} and~\eqref{sab1050f}, to obtain $u\in L^p(\Omega_m,\nu)$ such that
$$\  u(x)-\int_{\Omega_m}\a_p(x,y,u(y)-u(x))dm_x(y) =\phi(x),\quad x\in\Omega,
 $$
 and
$$-\int_{\Omega_m}\a_p(x,y,u(y)-u(x))dm_x(y)
=\varphi(x),\quad x\in \partial_m\Omega.$$
 Indeed, for $k\in\N$, since $\left(u_{n,k}\right)_n$ is bounded in $L^p(\Omega_m,\nu)$ we may find a subsequence (which we continue to denote by $\left(u_{n,k}\right)_n$) which converges weakly in $L^p(\Omega_m,\nu)$ to some $u^{\ast}_k\in L^p(\Omega_m,\nu)$. Now, since $\Vert u^\ast_k\Vert_p\leq \limsup_n\Vert u_{n,k}\Vert_p$ for every $k\in\N$, we may again find a subsequence of $(u^\ast_k)$ (which we denote equally) weakly convergent in $L^p(\Omega_m,\nu)$ to some $u\in L^p(\Omega_m,\nu)$. Note that, since $u_{n,k}$ is monotone in $n$ for every fixed $k$, we also have that $u_{n,k}\to_n u^\ast_k$ pointwise $\nu$-a.e. (the limits coincide by \cite[Theorem 1.35.]{Ambrosio}). Then, since $u_{n,k}$ is monotone in $k$ for every fixed $n$ we get that $(u^\ast_k)$ is monotone thus $u^\ast_k\to_k u$ pointwise $\nu$-a.e.
Moreover, $(|u_{n,k}|^{p-1}u_{n,k})$ is bounded in $L^1(\Omega_m,\nu)$ and is also monotone with respect to $n$ and $k$ in the same way as $(u_{n,k})$. Consequently, for a fixed $k\in\N$, by the monotone convergence theorem we have that $\vert u_{n,k}\vert^{p-1}u_{n,k}\to_n \vert u^\ast_k\vert^{p-1}u^\ast_k$ in $L^1(\Omega_m,\nu)$ and $\vert u^\ast_{k}\vert^{p-1}u^\ast_{k}\to_k \vert u\vert^{p-1}u$ in $L^1(\Omega_m,\nu)$. In particular,
 $$\int_{\Omega_m} ||u_{n,k}|^{p-1}u_{n,k}|d\nu\stackrel{n}{\longrightarrow} \int_{\Omega_m}|u^\ast_k|^pd\nu \ \ \hbox{ and } \ \ \int_{\Omega_m} ||u^\ast_{k}|^{p-1}u_{k}|d\nu\stackrel{k}{\longrightarrow} \int_{\Omega_m}|u|^pd\nu .$$
It follows, by the weak convergence $u_{n,k}\rightharpoonup_{n} u_k^{\ast}$ in $L^p(\Omega_m,\nu)$ together with the convergence of norms $\Vert u_{n,k} \Vert_{L^p(\Omega_m,\nu)}\to_n \Vert u_k^{\ast} \Vert_{L^p(\Omega_m,\nu)} $, that, for each $k\in\N$, $u_{n,k}\to_n u^\ast_k$ in $L^p(\Omega_m,\nu)$ and, similarly, that $u^\ast_k\to_k u$ in $L^p(\Omega_m,\nu)$.  Moreover, there exist $h_k\in L^p(\Omega_m,\nu)$, $k\in\N$, and $h\in L^p(\Omega_m,\nu)$ such that $|u_{n,k}|\leq h_k$ for every $n$, $k\in \N$ and $|u_k^\ast|\le h$ for every $k\in\N$.
 Finally, let us see that we can pass to the limit in \eqref{sab1050af} and \eqref{sab1050f}. Let $A\subset\Omega_m$ be a $\nu$-null set such that $|h_k(x)|<+\infty$, $|h(x)|<+\infty$, $u_{n,k}(x)\to_{n} u^\ast_k(x)<+\infty$, $u^\ast_{k}(x)\to_{k} u(x)<+\infty$ and $|u_{n,k}(x)|<+\infty$ for every $x\in \Omega_m\setminus A$ and $n$, $k\in\N$. Note that, since $m_x<<\nu$ for every $x\in X$, we also have that $A$ is $m_x$-null for every $x\in X$. Then, by \eqref{ll002}, there exists a $\nu\otimes m_x$-null set $\Psi \subset Q_1$ such that $\a_p(x,y,\cdot)$ is continuous for every $(x,y)\in Q_1\setminus\Psi$. Let $B\subset \Omega_m$ such that the section $\Psi_x$ of $\Psi$ is $m_x$-null for every $x\in X\setminus B$. Then,  $$\a_p(x,y,u_{n,k}(x)-u_{n,k}(y))\stackrel{n}{\longrightarrow}\a_p(x,y,u^\ast_k(x)-u^\ast_k(y))$$
  and
$$\a_p(x,y,u^\ast_{k}(x)-u^\ast_{k}(y))\stackrel{k}{\longrightarrow}\a_p(x,y,u(x)-u(y))$$
  pointwise for every $(x,y)\in\Psi\setminus (A\times A)$. Now,
 $$|\a_p(x,y,u_{n,k}(x)-u_{n,k}(y))|\le C(1+|u_{n,k}(x)-u_{n,k}(y)|^{p-1})\le \tilde{C}(1+|u_{n,k}(x)|^{p-1}+|u_{n,k}(y)|^{p-1}) $$
 $$\le \tilde{C}(1+|h_{k}(x)|^{p-1}+|h_{k}(y)|^{p-1})=:g_k(x,y)\in L^{\frac{p}{p-1}}(\Omega_m,\nu(dy))$$
 for every $x\in \Omega_m\setminus A$. Moreover, $g_k(x,y)\in L^{\frac{p}{p-1}}(\Omega_m,m_x(dy))\subset L^{1}(\Omega_m,m_x(dy))$ for $\nu$-a.e. $x\in \Omega_m$. So we may apply the dominated convergence theorem to get
 $$\int_{\Omega_m}\a_p(x,y, u_{n,k}(y)-u_{n,k}(x))dm_x(y)\to_{n} \int_{\Omega_m}\a_p(x,y, u^\ast_{k}(y)-u^\ast_{k}(x))dm_x(y)$$
 for $\nu$-a.e. $x\in\Omega_m$. Similarly, we can take limits in $k$ so that
  $$\int_{\Omega_m}\a_p(x,y, u^\ast_{k}(y)-u^\ast_{k}(x))dm_x(y)\to_{k} \int_{\Omega_m}\a_p(x,y, u(y)-u(x))dm_x(y)$$
 for $\nu$-a.e. $x\in\Omega_m$.

Therefore, the range condition in~\eqref{rangecond01} holds.
\end{proof}

\begin{theorem}\label{remdom01} Let $\varphi\in L^{p'}(\partial_m\Omega,\nu)$. Then,
 $$\overline{D(B^m_{\a_p,\varphi})}^{L^{p'}(\Omega,\nu)}=L^{p'}(\Omega,\nu).$$
\end{theorem}

\begin{proof}
   Let us see that, given  $z\in L^{\infty}(\Omega,\nu)$, $$u_n:=\left(I+\frac{1}{n}B^m_{\a_p,\varphi}\right)^{-1}z\to z \quad \hbox{in} \ L^{p'}(\Omega,\nu).$$
  Since $(u_n, n(z - u_n)) \in B^m_{\a_p,\varphi}$, we have
   \begin{equation}\label{eqq1}u_n(x)-z(x)=\frac1n\int_{\Omega_m}\a_p(x,y,u_n(y)-u_n(x))dm_x(y),\quad x\in\Omega.
 \end{equation}
 Hence,
 $$
 \begin{array}{c}
 \displaystyle \int_\Omega|u_n(x)-z(x)|^{p'}d\nu(x)
 =\frac{1}{n^{p'}}\int_\Omega\left|
 \int_{\Omega_m}\a_p(x,y,u_n(y)-u_n(x))dm_x(y)\right|^{p'}d\nu(x)
 \\ \\
 \displaystyle
 \le\frac{1}{n^{p'}}\int_\Omega
 \int_{\Omega_m}|\a_p(x,y,u_n(y)-u_n(x))|^{p'}dm_x(y) d\nu(x)
 \\ \\
 \displaystyle
 \le\frac{1}{n^{p'}}\int_{Q_1}|\a_p(x,y,u_n(y)-u_n(x)|^{p'}d(\nu\otimes m_x)(x,y).
 \end{array}
 $$
 Therefore, we only need to prove that
  $$\frac{1}{n^{p'}}\int_{Q_1}|\a_p(x,y,u_n(y)-u_n(x)|^{p'}d(\nu\otimes m_x)(x,y)\ \hbox{ converges to $0$ as $n\to+\infty$},
$$
for which, on account of~\eqref{llo1}, it is enough to see
that
\begin{equation}\label{unif001bis}\frac{1}{n^{p'}}\int_{\Omega_m}|u_n(x)|^{p}d\nu(x)\ \hbox{ converges to $0$ as $n\to+\infty$}.
  \end{equation}

Now, by Remark \eqref{porsi001}, we have that
$$\int_\Omega u_nd\nu=\int_\Omega zd\nu+\frac1n\int_{\partial_m\Omega}\varphi d\nu,
$$
  thus, to prove \eqref{unif001bis}, we only need to see that
  \begin{equation}\label{agost003niu001} \begin{array}{c}
 \displaystyle
  \frac1n\left\Vert u_{n}-\frac{1}{\nu(\Omega)}\int_\Omega u_{n}d\nu\right\Vert_{L^p(\Omega_m,\nu)}^p \quad\hbox{is uniformly bounded.}
\end{array}
 \end{equation}
  Multiplying \eqref{eqq1} by $ u_{n}-\frac{1}{\nu(\Omega)}\int_\Omega u_{n}d\nu$, integrating over $\Omega$ with respect to $\nu$ and applying integration by parts (Remark \ref{remmon} with $f=n(z-u_n)$) we obtain
  $$\left\Vert u_{n}-\frac{1}{\nu(\Omega)}\int_\Omega u_{n}\right\Vert_{L^2(\Omega,\nu)}^2 + \frac{1}{2n} \int_{Q_1} \a_p(x,y,u_n(y)-u_n(x))(u_n(y) - u_n(x))d(\nu\otimes m_x)(x,y) $$ $$= \frac{1}{n} \int_{\partial_m \Omega} \varphi \left(u_{n}-\frac{1}{\nu(\Omega)}\int_\Omega u_{n}d\nu \right) d\nu + \int_\Omega \left(z-\frac{1}{\nu(\Omega)}\int_\Omega u_{n}d\nu \right)\left(u_{n}-\frac{1}{\nu(\Omega)}\int_\Omega u_{n}d\nu \right) d\nu $$
   $$ = \frac{1}{n} \int_{\partial_m \Omega} \varphi \left(u_{n}-\frac{1}{\nu(\Omega)}\int_\Omega u_{n}d\nu \right) d\nu + \int_\Omega z\left(u_{n}-\frac{1}{\nu(\Omega)}\int_\Omega u_{n}d\nu \right) d\nu $$
   $$ \leq \frac1n\Vert\varphi\Vert_{L^{p'}(\partial_m\Omega,\nu)}\left\Vert u_{n}-\frac{1}{\nu(\Omega)}\int_\Omega u_{n}d\nu\right\Vert_{L^p(\Omega_m,\nu)}+ \frac12\Vert z\Vert_{L^{2}(\Omega,\nu)}^2 +\frac12\left\Vert u_{n}-\frac{1}{\nu(\Omega)}\int_\Omega u_{n}\right\Vert_{L^2(\Omega,\nu)}^2.$$
  On the other hand, by \eqref{llo2} and Poincar\'{e}'s inequality~\eqref{agost004}, we have
  \begin{gather*}
  \frac{1}{2n} \int_{Q_1} \a_p(x,y,u_n(y)-u_n(x))(u_n(y) - u_n(x))d(\nu\otimes m_x)(x,y)\\
  \geq \frac{c}{2\lambda^p n } \left\Vert u_{n}-\frac{1}{\nu(\Omega)}\int_\Omega u_{n}d\nu\right\Vert_{L^p(\Omega_m,\nu)}^p .
  \end{gather*}
  Therefore,
$$ \begin{array}{c}
 \displaystyle
  \frac{c}{2\lambda^p}\frac1n \left\Vert u_{n}-\frac{1}{\nu(\Omega)}\int_\Omega u_{n}d\nu\right\Vert_{L^p(\Omega_m,\nu)}^p
  \\ \\ \displaystyle \le      \frac12\Vert z\Vert_{L^{2}(\Omega,\nu)}^2
 + \frac1n\Vert\varphi\Vert_{L^{p'}(\partial_m\Omega,\nu)}\left\Vert u_{n}-\frac{1}{\nu(\Omega)}\int_\Omega u_{n}d\nu\right\Vert_{L^p(\Omega_m,\nu)},
\end{array}
$$
from where \eqref{agost003niu001} follows.
\end{proof}

\

The following theorem is a consequence of the previous results thanks to Theorem \ref{teointronls}.

\begin{theorem}\label{nsth01} Let $\varphi\in L^{p'}(\partial_m\Omega,\nu)$ and  $T>0$. For any
$u_0\in \overline{D(B^m_{\a_p,\varphi})}^{L^{p'}(\Omega,\nu)}=L^{p'}(\Omega,\nu)$  there exists a unique mild-solution $ u(t,x)$ of Problem~\eqref{Neumann4}.
 Moreover, for any $ q\geq p'$ and  $u_{0i}\in L^q(\Omega,\nu)$,  $i=1,2$, we have the following contraction principle for the corresponding mild-solutions $u_i$:
$$  \Vert (u_1(t,.)-u_2(t,.))^+\Vert_{L^q(\Omega,\nu)}\le \Vert (u_{0,1}-u_{0,2})^+\Vert_{L^q(\Omega,\nu)}
 \quad \hbox{for any \  $0\le t< T$.}$$

 If $u_0\in D(B^m_{\a_p,\varphi})$ then the mild-solution is a strong solution.
\end{theorem}

It is natural to ask whether $u\in L^\infty(\Omega_m,\nu)$ whenever $u$ is the solution of the problem
$$
\left\{ \begin{array}{ll} u(x)-\hbox{div}_m\a_p u(x) = v(x), \quad & x \in\Omega, \\ [8pt] \mathcal{N}^{\a_p}_1 u(x) = \varphi(x), \quad & x\in\partial_m\Omega, \end{array} \right.
$$
with $v\in L^\infty(\Omega,\nu)$ and  $\varphi\in  L^\infty(\partial_m\Omega,\nu)$. In the next example we will see that this is not true in general and, as a consequence, that there exist metric random walk spaces that do not satisfy a Poincar\'{e} type inequality like~\eqref{agost004}.

\begin{example}{\rm
Let $V(G):=\{x_0, x_1,\ldots, x_n, \ldots\}$, $w_{x_0, x_n}=w_{x_n, x_0}=\frac{1}{7^n}$ for $n \in \N$, $w_{x_n, x_n}=\frac{1}{3^n}-\frac{1}{7^n}$ for $n \in \N $ and $w_{x,y}=0$ otherwise. Consider the metric random walk space $[V(G),d_G,m^G]$ associated to this infinite weighted discrete graph. Note that this graph is not locally finite. Then,
$$d_{x_0} = \frac{1}{6}, \quad d_{x_n} = \frac{1}{3^n}, \quad  m_{x_0}=\sum_{n\geq 1} \frac{6}{7^n}\delta_{x_n} ,$$
$$m_{x_n}=\left(\frac{3}{7}\right)^n\delta_{x_0}+\left(1-\left(\frac{3}{7}\right)^n\right)\delta_{x_n} \ , \quad n\geq 1,$$
and
$$\nu=\frac16\delta_{x_0}+\sum_{n\geq 1} \frac{1}{3^n}\delta_{x_n}, \quad \nu(V) = \frac16 + \sum_{n=1}^\infty \frac{1}{3^n} = \frac{2}{3} .$$
Let $1< p < +\infty$ and $\Omega:=\{x_0\}$, and denote $m:=m^G$, so that $\partial_m \Omega =\{x_1,\ldots, x_n,\ldots\}$.

Let $a_p(x,y,r)=|r|^{p-2}r$, define $u:\Omega_m\rightarrow \R$ by
$$u(x):=\left\{\begin{array}{ll}
    0 & \hbox{ if } x=x_0 \\
    2^\frac{n}{p-1} & \hbox{ if } x=x_n, \ n\geq 1,
  \end{array}\right.$$
$v:\Omega\rightarrow \R$ by $v(x_0)=-\frac{12}{5}$ and $\varphi:\partial_m\Omega\rightarrow \R$ by $\varphi(x_n)=(\frac67)^n$, $n\geq 1$. Then $u,v\in  L^\infty(\Omega,\nu)$, $\varphi\in L^\infty(\partial_m\Omega,\nu)$ and $(x,y)\mapsto a_p(x,y,u(y)-u(x))\in L^{p'}(Q_1,\nu\otimes m_x)$.

Now,
$$u(x_0)-\int_{\Omega_m}\a_p(x_0, y, u(y)-u(x_0))dm_{x_0}(y)$$
$$=u(x_0)-\sum_{n\ge 1}|u(x_n)-u(x_0)|^{p-2}(u(x_n)-u(x_0))m_{x_0}(\{x_n\})$$
$$=u(x_0)-\sum_{n\ge 1}u(x_n)^{p-1}\frac{6}{7^n}=-\sum_{n\ge 1}2^n\frac{6}{7^n}=-\frac{12}{5}=v(x_0)$$
and
$$-\int_{\Omega_m}\a_p(x_n, y, u(y)-u(x_n))dm_{x_n}(y)=-|u(x_0)-u(x_n)|^{p-2}(u(x_0)-u(x_n))m_{x_n}(\{x_0\})=$$
$$=u(x_n)^{p-1}\left(\frac{3}{7}\right)^n=\left(\frac{6}{7}\right)^n=\varphi(x_n).$$
Therefore, $u$ is a solution of the Neumann problem
$$
\left\{ \begin{array}{ll} u-\hbox{div}_m\a_p u = v \quad &\hbox{in} \ \ \Omega, \\ [8pt] \mathcal{N}^{\a_p}_m u = \varphi \quad &\hbox{in} \ \ \partial_m\Omega. \end{array} \right.
$$
Note that $v\in L^\infty(\Omega,\nu)$ and $\varphi\in  L^\infty(\partial_m\Omega,\nu)$ but $u\not\in L^\infty(\Omega_m, \nu)$. Note also that $u\in L^p(\Omega_m,\nu)$ for sufficiently large $p$ ($p>\frac{1}{\log(\frac32)}$).

Consequently, since $v\in L^\infty(\Omega,\nu)$ and $\varphi\in  L^\infty(\partial_m\Omega,\nu)$ but, for $p\leq\frac{1}{\log(\frac32)}$,
$$u = (I + B^m_{\a_p,\varphi})^{-1} v $$
does not belong to $L^p(\Omega_m,\nu)$, we have that this metric random walk space does not satisfy a Poincar\'{e} type inequality like~\eqref{agost004}.}
\end{example}

\subsection{Neumann boundary conditions of   Dipierro--Ros-Oton--Valdinoci type}\label{aus24}

  In this subsection we continue to work under Assumptions~\ref{as1} and \ref{as2}. However, we do not require a Poincar\'{e} type inequality.

 \begin{definition}{\rm
Let
$$L^{m, \infty}(\partial_m\Omega,\nu):=\left\{\varphi:\partial_m\Omega\rightarrow \R \ : \ \varphi \hbox{ is $\nu$-measurable and } \frac{\varphi }{m_{(.)}(\Omega)}\in L^\infty(\partial_m\Omega, \nu) \right\} . $$
}
\end{definition}

\begin{remark}{
Note that $L^{m, \infty}(\partial_m\Omega,\nu)\subset L^{\infty}(\partial_m\Omega,\nu)$.

Suppose that $[V,d_G,m^G]$ is the metric random walk space associated to a locally finite weighted discrete graph as described in Example \ref{graphs1} and let $\Omega\subset V$. Then, if $\partial_m\Omega\subset V$ is a finite set, we have that $L^{m,\infty}(\partial_m\Omega,\nu)= L^{\infty}(\partial_m\Omega,\nu)$.

Consider now the metric random walk space $[\R^N, d, m^J]$ given in Example \ref{graphs101}. Let $\Omega\subset\R^N$ be a bounded domain and denote
$$\Omega_r:=\{x\in\R^N : \hbox{dist}(x,\Omega)<r\} . $$
Suppose that $\hbox{supp}(J)\supseteq B(0,R)$. Then,
$$\left\{\varphi\in L^{\infty}(\partial_m\Omega,\nu) \ : \ \hbox{supp}(\varphi)\subset \Omega_r, \ r<R \right\}\subset L^{m,\infty}(\partial_m\Omega,\nu) .$$
Indeed, let $\varphi\in L^{\infty}(\partial_m\Omega,\nu)$ such that $\hbox{supp}(\varphi)\neq \emptyset$  and $\hbox{supp}(\varphi)\subset \Omega_r$ for some $r<R$. It is enough to see that there exists $\delta>0$ such that
$$m_x^J(\Omega)=\int_\Omega J(x-y) dy \ge \delta >0$$
for every $x\in\hbox{supp}(\varphi)$. Suppose otherwise that there exists a sequence $(x_n)\subset \hbox{supp}(\varphi)$ such that $\lim_n \int_\Omega J(x_n-y) dy=0$, then, since $\partial_m\Omega$ is bounded, there exists a subsequence of $(x_n)$ converging to $x_0\in \hbox{supp}(\varphi)$. Therefore, by the continuity of $J$ and applying Fatou's Lemma we get that $\int_\Omega J(x_0-y) dy=0$. However, this is not possible because $\hbox{dist}(x_0,\Omega)\le r<R$ and, therefore, since $\Omega$ is open, we have that $\mathcal{L}^N(B(x_0,R)\cap\Omega)>0$ with $B(x_0,R)\cap\Omega\subset\hbox{supp}(J(x_0-.))$ so
$$\int_\Omega J(x_0-y) dy\ge \int_{B(x_0,R)\cap\Omega} J(x_0-y) dy>0 . $$
In particular, characteristic functions of sets $A\subset\Omega_r$ with $r<R$, belong to  $L^{m,\infty}(\partial_m\Omega,\nu)$.
}
\end{remark}

In this subsection we study the  problem
\begin{equation}\label{NNeumann4}
\left\{ \begin{array}{ll} u_t(t,x) = \hbox{div}_m\a_p u(t,x),    &x\in  \Omega,\ 0<t<T, \\ \\  {\mathcal{N}^{\a_p}_2} u(t,x) = \varphi(x),    &x\in\partial_m\Omega, \  0<t<T, \\ \\ u(0,x) = u_0(x),    &x\in \Omega. \end{array} \right.
\end{equation}
 To this end, we define the  following operator in   $ L^1(\Omega,\nu)\times L^1(\Omega,\nu)$ associated with the problem.

\begin{definition}{\rm  Let $1 < p < \infty$. Let $\varphi\in L^{m, \infty}(\partial_m\Omega,\nu)$.  We say that $(u,v) \in A^m_{\a_p,\varphi}$ if $ u,v \in L^1(\Omega,\nu)$,
  and there exists a $\nu$-measurable function $\overline u$ in $\Omega_m$  with $\overline{u}_{\vert \Omega} = u$ (that we denote equally as $u$) satisfying
 $$  m_{(\cdot)} (\Omega)|u|^{p-1}\in L^1(\partial_m\Omega,\nu),
$$
$$
(x,y)\mapsto a_p(x,y,u(y)-u(x))\in L^{1}(Q_2,\nu\otimes m_x),
$$
and
\begin{equation}\label{def1538}
\left\{ \begin{array}{ll} -\hbox{div}_m\a_p u = v \quad &\hbox{in} \ \ \Omega, \\ [8pt] \mathcal{N}^{\a_p}_2 u = \varphi \quad &\hbox{in} \ \ \partial_m\Omega, \end{array} \right.
\end{equation}
that is,
$$
v(x) = -  \int_{\Omega_m} \a_p(x,y,u(y)-u(x)) dm_x(y), \quad x \in \Omega,
$$
and
$$
\varphi(x) = -\int_{\Omega} \a_p(x,y,u(y)-u(x)) dm_x(y), \quad x \in \partial_m\Omega.
$$
}
\end{definition}

\begin{remark}\label{elrem001} { Let  $\varphi\in L^{m, \infty}(\partial_m\Omega,\nu)$.

\item{ 1.}  Let   $u\in L^{p}(\Omega,\nu)$ and
let $\Psi:\R\rightarrow \R$ be defined by
$$\Psi(r):=-\int_\Omega a_p(x,y,u(y)-r)dm_x(y). $$
Then, since $\Psi$ is increasing by \eqref{llo3}, the equation
\begin{equation}\label{dominioLinfinito}
-\int_\Omega \a_p(x,y,u(y)-r)dm_x(y)=\varphi(x), \quad x\in\partial_m\Omega,
\end{equation}
 has a  unique solution $r=:\overline{u}(x)$, which is easily seen to be $\nu$-measurable.
\medskip
\item{ 2.} As a consequence, the extension of $u$ to the boundary $\partial_m \Omega$ in Definition~\ref{def1538}   is unique.
\medskip
\item{ 3.}  Let us see that, if $u\in L^{\infty}(\Omega,\nu)$, then
$$\Vert {\overline{u}}\Vert_{L^\infty(\partial_m\Omega,\nu)}\leq \Vert u\Vert_{L^\infty(\Omega, \nu)}+\frac{1}{c^{\frac{1}{p-1}}}\left\Vert \frac{\varphi}{m_{(.)}(\Omega)} \right\Vert_{L^\infty(\partial_m \Omega, \nu)}^{\frac{1}{p-1}} \,.
$$
 Indeed, let us denote $u(x):=\overline{u}(x)$, $x\in\partial_m\Omega$, and  suppose that $\Vert u\Vert_{L^\infty(\partial_m\Omega,\nu)}> \Vert u\Vert_{L^\infty(\Omega,\nu)}$, otherwise the result is trivial. Let $0<\varepsilon<\Vert u\Vert_{L^\infty(\partial_m\Omega,\nu)}-\Vert u\Vert_{L^\infty(\Omega,\nu)}$,
$$A^+:=\{x\in\partial_m\Omega : u (x)>\Vert u\Vert_{L^\infty(\partial_m\Omega,\nu)}-\varepsilon\}$$
and
$$A^-:=\{x\in\partial_m\Omega : u (x)<-\Vert u\Vert_{L^\infty(\partial_m\Omega,\nu)}+\varepsilon\} .$$
Suppose first that $\nu(A^-)>0$ and let
$$B:=\{y\in \Omega : u(y)<-\Vert u\Vert_{L^\infty(\Omega,\nu)}\} .$$
Then, since $\nu(B)=0$ and $m_x<<\nu$ for every $x\in X$, we have that $m_x(B)=0$, $x\in X$, and, consequently, $\nu\otimes m_x (A^-\times B)=0$. Now,
$$u(y)-u(x)>\Vert u\Vert_{L^\infty(\partial_m\Omega,\nu)}-\Vert u\Vert_{L^\infty(\Omega,\nu)}-\varepsilon>0$$
for every $(x,y)\in A^-\times (\Omega\setminus B)$. Therefore, since, by \eqref{llo2}, $\a_p(x,y,r)\geq cr^{p-1}$ for $r\geq 0$, we have that
$$\a_p(x,y,u(y)-u(x))\geq c(\Vert u\Vert_{L^\infty(\partial_m\Omega,\nu)}-\Vert u\Vert_{L^\infty(\Omega,\nu)}-\varepsilon)^{p-1}$$
for every $(x,y)\in A^-\times (\Omega\setminus B)$. Now, integrating \eqref{dominioLinfinito} over $A^-$, we get:
$$\begin{array}{l}\displaystyle
-\int_{A^-}\int_{\Omega}\a_p(x,y,u(y)-u(x))dm_x(y)d\nu(x)
\qquad\qquad \qquad\qquad\\ \\
\displaystyle \qquad \quad
\geq\int_{A^-}\varphi(x)d\nu(x)=\int_{A^-}\frac{\varphi(x)}{m_x(\Omega)}m_x(\Omega)d\nu(x) \\ \\ \qquad\qquad \ \displaystyle
\geq -\left\Vert \frac{\varphi}{m_{(.)}(\Omega)} \right\Vert_{L^\infty(\partial_m \Omega,\nu)}\int_{A^-}m_x(\Omega)d\nu(x)  .
\end{array}$$
but, by the previous computations,
$$-\int_{A^-}\int_{\Omega}\a_p(x,y,u(y)-u(x))dm_x(y)d\nu(x)$$
$$\leq -c(\Vert u\Vert_{L^\infty(\partial_m\Omega,\nu)}-\Vert u\Vert_{L^\infty(\Omega,\nu)}-\varepsilon)^{p-1} \int_{A^-}m_x(\Omega)d\nu(x)$$
thus
$$c(\Vert u\Vert_{L^\infty(\partial_m\Omega,\nu)}-\Vert u\Vert_{L^\infty(\Omega,\nu)}-\varepsilon)^{p-1}\le \left\Vert \frac{\varphi}{m_{(.)}(\Omega)} \right\Vert_{L^\infty(\partial_m \Omega,\nu)} $$
and the result follows since $\varepsilon$ was arbitrarily small. If $\nu(A^-)=0$ then $\nu(A^+)>0$ and we would proceed analogously.}

\end{remark}

  \begin{theorem}\label{CAandRangeCond} Let  $\varphi\in L^{m, \infty}(\partial_m\Omega,\nu)$. The  operator $A^m_{\a_p,\varphi}$ is  completely accretive and satisfies  the range condition
\begin{equation}\label{rangecond01N}
 L^{p'}(\Omega,\nu) \subset R(I+ A^m_{\a_p,\varphi}).
\end{equation}

\end{theorem}
\begin{proof}

The proof of the complete accretivity of $A^m_{\a_p,\varphi}$ follows similarly to that of $B^m_{\a_p,\varphi}$.

  Let us see that $A^m_{\a_p,\varphi}$ satisfies the range condition~\eqref{rangecond01N}, that is, let us prove that, for $\phi\in L^{p'}(\Omega,\nu)$, there exists $u\in D(A^m_{\a_p,\varphi})$ such  that $$u+A^m_{\a_p,\varphi}u\ni\phi.$$
  We divide the proof into two steps.

\noindent {\bf Step 1}.  Assume that $\phi\in L^\infty(\Omega,\nu)$.
Working as in the proof of Theorem \ref{teo27} but defining
 $A_2$ by
 $$A_2(u)(x):= -\int_{\Omega}\a_p(x,y,u(y)-u(x))dm_x(y)
+\frac{1}{n}|u(x)|^{p-2}u^+(x)
-\frac{1}{k}|u(x)|^{p-2}u^-(x),$$
for $x\in \partial_m\Omega$, we have that,
  for $k,n\in \mathbb{N}$ and $K>0$,
 there exist $u_{n,k}\in L^{p}(\Omega_m,\nu)$ such that
 $$\begin{array}{l}
\displaystyle T_K(u_{n,k})(x)-\int_{\Omega_m}\a_p(x,y,u_{n,k}(y)-u_{n,k}(x))dm_x(y) \qquad\qquad  \\
\displaystyle \qquad\qquad +\frac{1}{n}|u_{n,k}(x)|^{p-2}u_{n,k}^+(x)
-\frac{1}{k}|u_{n,k}(x)|^{p-2}u_{n,k}^-(x)=\phi(x),
\end{array}$$
$x\in\Omega$, and
$$  -\int_{\Omega}\a_p(x,y,u_{n,k}(y)-u_{n,k}(x))dm_x(y)
+\frac{1}{n}|u_{n,k}(x)|^{p-2}u_{n,k}^+(x)
-\frac{1}{k}|u_{n,k}(x)|^{p-2}u_{n,k}^-(x)=\varphi(x),$$
$x\in \partial_m\Omega$. Now, let $M>0$. Multiplying the first equation by $(u_{n,k}-M)^+$ and integrating over $\Omega$ with respect to $\nu$, by Proposition~\ref{prop23}, we get  that, after removing some positive terms,
$$
\begin{array}{c}\displaystyle
\int_{\Omega}T_K(u_{n,k})(u_{n,k}-M)^+ d\nu+ \frac{1}{n}\int_{\partial_m\Omega}|u_{n,k}|^{p-2}u_{n,k}^+(u_{n,k}-M)^+d\nu
\\ \\
\displaystyle \le
\int_{\Omega}\phi (u_{n,k}-M)^+d\nu + \int_{\partial_m\Omega}\varphi (u_{n,k}-M)^+d\nu .
\end{array}
$$
Therefore, taking
$$M=M_{\phi,\varphi,n,k}:=\max\left\{\Vert \phi\Vert_{L^\infty(\Omega,\nu)},\left(n\Vert \varphi\Vert_{L^\infty(\partial_m\Omega,\nu)}\right)^\frac{1}{p-1},\left(k\Vert \varphi\Vert_{L^\infty(\partial_m\Omega,\nu)}\right)^\frac{1}{p-1}\right\}$$
we get that
$$\begin{array}{c}
\displaystyle\int_{\Omega}\big(T_K(u_{n,k})-M\big)(u_{n,k}-M)^+d\nu + \frac{1}{n}\int_{\partial_m\Omega}\big(|u_{n,k}|^{p-2}u_{n,k}^+-M^{p-1})(u_{n,k}-M)^+d\nu
\\ \\ \displaystyle
\le
\int_{\Omega}(\phi-M) (u_{n,k}-M)^+d\nu + \int_{\partial_m\Omega}(\varphi-\frac1n M^{p-1}) (u_{n,k}-M)^+d\nu\le 0
\end{array}$$
 and, consequently, taking $K>M$, we get that
$$u_{n,k}\le M\quad\nu-\hbox{a.e. in }\Omega_m;$$
and, similarly, we get that
$$u_{n,k}\ge -M\quad\nu-\hbox{a.e. in }\Omega_m.$$
Hence,
$$\Vert u_{n,k}\Vert_{L^\infty(\Omega_m,\nu)}\le M  .$$
Therefore, we have
\begin{equation}\label{sab1050a}\begin{array}{c}\displaystyle u_{n,k}(x)-\int_{\Omega_m}\a_p(x,y, u_{n,k}(y)-u_{n,k}(x))dm_x(y)+ \qquad\qquad \qquad\qquad\\ \\
\displaystyle \qquad\qquad \qquad\qquad+
\frac{1}{n}|u_{n,k}(x)|^{p-2}u_{n,k}^+(x)
-\frac{1}{k}|u_{n,k}(x)|^{p-2}u_{n,k}^-(x)=\phi(x),
\end{array}\end{equation}
for $x\in\Omega$;
and
\begin{equation}\label{sab1050}\begin{array}{c}\displaystyle
-\int_{\Omega}\a_p(x,y, u_{n,k}(y)-u_{n,k}(x))dm_x(y)+
\qquad\qquad \qquad\qquad\\ \\
\displaystyle \qquad\qquad \qquad\qquad
+\frac{1}{n}|u_{n,k}(x)|^{p-2}u_{n,k}^+(x)
-\frac{1}{k}|u_{n,k}(x)|^{p-2}u_{n,k}^-(x)=\varphi(x),
\end{array}\end{equation}
for $x\in \partial_m\Omega$.

Let us now see that $\Vert u_{n,k}\Vert_{L^\infty(\Omega_m,\nu)}$ is uniformly bounded in $n$ and $k$. First, working as in the proof of Remark~\ref{elrem001}.{\it 3}, we prove  that
$$\Vert u_{n,k}\Vert_{L^\infty(\partial_m\Omega,\nu)}\leq \Vert u_{n,k}\Vert_{L^\infty(\Omega,\nu)}+\frac{1}{c^{\frac{1}{p-1}}}\left\Vert \frac{\varphi}{m_{(.)}(\Omega)} \right\Vert_{L^\infty(\partial_m \Omega,\nu)}^{\frac{1}{p-1}}
$$
for every $n$, $k\ge 1$. Indeed, define $A^-$ as in that remark  and integrate \eqref{sab1050} over $A^-$ with respect to $\nu$ (note that the term involving $\frac{1}{n}|u_{n,k}(x)|^{p-2}u_{n,k}^+(x)
-\frac{1}{k}|u_{n,k}(x)|^{p-2}u_{n,k}^-(x)$ does not affect the reasoning). The same can be done with $A^+$.
Therefore, it is enough to see that $\Vert u_{n,k}\Vert_{L^\infty(\Omega,\nu)}$ is uniformly bounded in $n$ and $k$. Let
  $$K:= \frac{1}{c^{\frac{1}{p-1}}}\left\Vert \frac{\varphi}{m_{(.)}(\Omega)} \right\Vert_{L^\infty(\partial_m \Omega,\nu)}^{\frac{1}{p-1}} , $$
 so that $\Vert u_{n,k}\Vert_{L^\infty(\partial_m\Omega,\nu)}\leq \Vert u_{n,k}\Vert_{L^\infty(\Omega,\nu)}+K$. Now, if all of the $u_{n,k}$ are $\nu$-null the result is trivial. Therefore, fix some $u_{n,k}\not\equiv 0$ and $0<\varepsilon<\Vert u_{n,k}\Vert_{L^\infty(\Omega,\nu)}$. Let
   $$A_\star^+:=\{x\in \Omega : u_{n,k}(x)>\Vert u_{n,k}\Vert_{L^\infty(\Omega,\nu)}-\varepsilon\}$$
and
$$A_\star^-:=\{x\in \Omega : u_{n,k}(x)<-\Vert u_{n,k}\Vert_{L^\infty(\Omega,\nu)}+\varepsilon\} .$$
Suppose first that $\nu(A_\star^+)>0$. Integrating over $A_\star^+$ in \eqref{sab1050a} we get:
$$\begin{array}{l}\displaystyle \int_{A_\star^+}u_{n,k}(x)d\nu(x)-\int_{A_\star^+}\int_{\Omega_m}\a_p(x,y,u_{n,k}(y)-u_{n,k}(x))dm_x(y)d\nu(x) \\ \\
\displaystyle \qquad\qquad\qquad+
\frac{1}{n}\int_{A_\star^+}|u_{n,k}(x)|^{p-2}u_{n,k}^+(x)d\nu(x)
=\int_{A_\star^+}\phi(x)d\nu(x)\leq \nu(A_\star^+)\Vert \phi \Vert_{L^\infty(\Omega,\nu)}.
\end{array} $$
Consequently, dividing by $\nu(A_\star^+)$, we have
$$\Vert u_{n,k}\Vert_{L^\infty(\Omega,\nu)}-\varepsilon-\frac{1}{\nu(A_\star^+)}\int_{A_\star^+}\int_{\Omega_m}\a_p(x,y,u_{n,k}(y)-u_{n,k}(x))dm_x(y)d\nu(x) \leq \Vert \phi \Vert_{L^\infty(\Omega,\nu)} .$$
Now, $\a_p(x,y,u_{n,k}(y)-u_{n,k}(x))\leq \a_p(x,y,K+\varepsilon)\leq C \left(1+ (K+ \varepsilon)^{p-1}\right)$ for $\nu\otimes m$-a.e. $(x,y)\in A_\star^+\times \Omega_m$, thus, since $m_x(\Omega_m)=1$ for $\nu$-a.e. $x\in \Omega$,
$$
\frac{1}{\nu(A_\star^+)}\int_{A_\star^+}
\int_{\Omega_m}\a_p(x,y,u_{n,k}(y)-u_{n,k}(x))dm_x(y)d\nu(x)
\leq C\left(1+(K+\varepsilon)^{p-1}\right), $$
and, since $\varepsilon>0$ is arbitrarily small, we conclude that
$$\Vert u_{n,k}\Vert_{L^\infty(\Omega,\nu)} \leq \Vert \phi \Vert_{L^\infty(\Omega,\nu)} + C\left(1+\frac{1}{c}\left\Vert \frac{\varphi}{m_{(.)}(\Omega)} \right\Vert_{L^\infty(\partial_m \Omega,\nu)}\right)
$$
where the right hand side does not depend on $n$ or $k$. If $\nu(A_\star^+)=0$ then $\nu(A_\star^-)>0$ and we proceed similarly, that is, integrate over $A_\star^-$ in \eqref{sab1050a} to obtain that
$$\begin{array}{l}\displaystyle \int_{A_\star^-}u_{n,k}(x)d\nu(x)-\int_{A_\star^-}\int_{\Omega_m}\a_p(x,y,u_{n,k}(y)-u_{n,k}(x))dm_x(y)d\nu(x) \\ \\
\displaystyle \qquad\qquad\qquad-
\frac{1}{k}\int_{A_\star^-}|u_{n,k}(x)|^{p-2}u_{n,k}^-(x)d\nu(x)
=\int_{A_\star^-}\phi(x)d\nu(x)\geq -\nu(A_\star^-)\Vert \phi \Vert_{L^\infty(\Omega,\nu)}.
\end{array} $$
Then, dividing by $\nu(A_\star^-)$, we have
$$-\Vert u_{n,k}\Vert_{L^\infty(\Omega,\nu)}+\varepsilon-\frac{1}{\nu(A_\star^-)}\int_{A_\star^-}\int_{\Omega_m}\a_p(x,y,u_{n,k}(y)-u_{n,k}(x))dm_x(y)d\nu(x) \geq -\Vert \phi \Vert_{L^\infty(\Omega,\nu)} $$
which, using \eqref{llo4}, is equivalent to
$$\Vert u_{n,k}\Vert_{L^\infty(\Omega,\nu)} \leq \Vert \phi \Vert_{L^\infty(\Omega,\nu)}+\varepsilon+\frac{1}{\nu(A_\star^-)}\int_{A_\star^-}\int_{\Omega_m}\a_p(x,y,u_{n,k}(x)-u_{n,k}(y))dm_x(y)d\nu(x). $$
Now, $\a_p(x,y,u_{n,k}(x)-u_{n,k}(y))\leq \a_p(x,y,K+\varepsilon)\leq C \left(1+ (K+ \varepsilon)^{p-1}\right)$ for $\nu\otimes m$-a.e. $(x,y)\in A_\star^-\times \Omega_m$ and we conclude as before.

  Now, let us see that $\{u_{n,k}\res\Omega\} $ is $\nu$-a.e. nondecreasing in $n$, and $\nu$-a.e.  nonincreasing in $k$.  Let $n'<n$. Multiplying~\eqref{sab1050a} for $u_{n',k}$ and $u_{n,k}$ by $(u_{n',k}-u_{n,k})^+$, integrating over $\Omega$ with respect to $\nu$, and subtracting we obtain
$$ \begin{array}{l}\displaystyle \int_{\Omega}\left(u_{n',k}(x)-u_{n,k}(x)\right)\left(u_{n',k}(x)-u_{n,k}(x)\right)^+d\nu(x)
\\ \\
\displaystyle +\int_{\Omega}\left(\frac{1}{n'}|u_{n',k}(x)|^{p-2}u_{n',k}^+(x)-\frac{1}{n}|u_{n,k}(x)|^{p-2}u_{n,k}^+(x)\right)\left(u_{n',k}(x)-u_{n,k}(x)\right)^+d\nu(x)-\\ \\
\displaystyle-\int_{\Omega}\frac{1}{k}\left(|u_{n',k}(x)|^{p-2}u_{n',k}^-(x)-|u_{n,k}(x)|^{p-2}u_{n,k}^-(x)\right)\left(u_{n',k}(x)-u_{n,k}(x)\right)^+d\nu(x)
\\  \\
\displaystyle
-\int_\Omega\int_{\Omega_m}\a_p(x,y, u_{n',k}(y)-u_{n',k}(x))\left(u_{n',k}(x))-u_{n,k}(x)\right)^+dm_x(y)d\nu(x)
\\  \\
\displaystyle
+\int_\Omega\int_{\Omega_m}\a_p(x,y, u_{n,k}(y)-u_{n,k}(x))\left(u_{n',k}(x))-u_{n,k}(x)\right)^+dm_x(y)d\nu(x)
\\  \\
\displaystyle
=\int_\Omega \phi(x)d\nu(x)-\int_\Omega \phi(x)d\nu(x)=0 .
\end{array}
$$
Now, by Proposition~\ref{prop23} with $ w(x)=(u_{n',k}(x)-u_{n,k}(x))^+$ and recalling \eqref{llo3} (see also Remark~\ref{remmon}), and then using \eqref{sab1050} we obtain
$$\begin{array}{l}\displaystyle
-\int_\Omega\int_{\Omega_m}\a_p(x,y, u_{n',k}(y)-u_{n',k}(x))\left(u_{n',k}(x))-u_{n,k}(x)\right)^+dm_x(y)d\nu(x)
\\  \\
\qquad \displaystyle
+\int_\Omega\int_{\Omega_m}\a_p(x,y, u_{n,k}(y)-u_{n,k}(x))\left(u_{n',k}(x))-u_{n,k}(x)\right)^+dm_x(y)d\nu(x)
 \\ \\
 \displaystyle
=\frac12\int_{Q_2}(\a_p(x,y, u_{n',k}(y)-u_{n',k}(x))-\a_p(x,y, u_{n,k}(y)-u_{n,k}(x)))\times
\\  \\
\displaystyle
\qquad \qquad \qquad  \qquad\left((u_{n',k}(y))-u_{n,k}(y))^+-(u_{n',k}(x)-u_{n,k}(x))\right)^+dm_x(y)d\nu(x)
\\  \\
\qquad\displaystyle
-\int_{\partial_m \Omega}\int_{\Omega_m}(\a_p(x,y, u_{n',k}(y)-u_{n',k}(x))-\a_p(x,y, u_{n,k}(y)-u_{n,k}(x)))\times
\\  \\
\displaystyle
\qquad  \qquad\qquad  \qquad\quad\qquad  \qquad\qquad  \qquad\qquad  \left(u_{n',k}(x))-u_{n,k}(x)\right)^+dm_x(y)d\nu(x)
\\  \\
\displaystyle
\ge\int_{\partial_m\Omega}\left( \frac{1}{n'} |u_{n',k}(x)|^{p-2}u_{n',k}^+(x)-\frac{1}{n}|u_{n,k}(x)|^{p-2}u_{n,k}^+(x) \right)\left(u_{n',k}(x)-u_{n,k}(x)\right)^+d\nu(x)-
\\ \\
\displaystyle
\qquad-\frac{1}{k}\int_{\partial_m\Omega}\left(  |u_{n',k}(x)|^{p-2}u_{n',k}^-(x)-|u_{n,k}(x)|^{p-2}u_{n,k}^-(x) \right)\left(u_{n',k}(x)-u_{n,k}(x)\right)^+d\nu(x)
  \,. \qquad
\end{array}
$$
Consequently,
$$
 \begin{array}{l}\displaystyle \int_{\Omega}\left(u_{n',k}(x)-u_{n,k}(x)\right)\left(u_{n',k}(x)-u_{n,k}(x)\right)^+d\nu(x)
\\ \\
\displaystyle +\int_{\Omega}\left(\frac{1}{n'}|u_{n',k}(x)|^{p-2}u_{n',k}^+(x)-\frac{1}{n}|u_{n,k}(x)|^{p-2}u_{n,k}^+(x)\right)\left(u_{n',k}(x)-u_{n,k}(x)\right)^+d\nu(x)-\\ \\
\displaystyle-\int_{\Omega}\frac{1}{k}\left(|u_{n',k}(x)|^{p-2}u_{n',k}^-(x)-|u_{n,k}(x)|^{p-2}u_{n,k}^-(x)\right)\left(u_{n',k}(x)-u_{n,k}(x)\right)^+d\nu(x)
\\  \\
\displaystyle
+ \int_{\partial_m\Omega}\left( \frac{1}{n'} |u_{n',k}(x)|^{p-2}u_{n',k}^+(x)-\frac{1}{n}|u_{n,k}(x)|^{p-2}u_{n,k}^+(x) \right)\left(u_{n',k}(x)-u_{n,k}(x)\right)^+d\nu(x)-
\\ \\
\displaystyle
-\frac{1}{k}\int_{\partial_m\Omega}\left(  |u_{n',k}(x)|^{p-2}u_{n',k}^-(x)-|u_{n,k}(x)|^{p-2}u_{n,k}^-(x) \right)\left(u_{n',k}(x)-u_{n,k}(x)\right)^+d\nu(x)
 \le 0 \,. \qquad
\end{array}
$$
 Therefore, since the last four summands on the left hand side are non-negative we get that
$$\int_{\Omega}\left(u_{n',k}(x)-u_{n,k}(x)\right)\left(u_{n',k}(x)-u_{n,k}(x)\right)^+d\nu(x)\leq 0$$
so $u_{n,k}\res\Omega$ is $\nu$-a.e. nondecreasing in $n$. Similarly, we get that $u_{n,k}\res\Omega$ is $\nu$-a.e. nonincreasing in $k$.

 Let us see that these monotonicities also hold in $\partial_m\Omega$. Recall that
$$-\int_{\Omega}\a_p(x,y,u_{n,k}(y)-u_{n,k}(x))dm_x(y)
=\varphi(x)- \frac{1}{n}|u_{n,k}(x)|^{p-2}u_{n,k}^+(x)
+\frac{1}{k}|u_{n,k}(x)|^{p-2}u_{n,k}^-(x),$$
for $x\in \partial_m\Omega$.  Now, let $N\subset X$ be a $\nu$-null set such that, for every $x\in X\setminus N$,
$$(\a_p(x,y,r)-\a_p(x,y,s))(r-s) > 0 \quad \hbox{for $m_x$-a.e. $y\in X$ and for all $r\neq s$.}$$
Then, for a fixed $k\in\N$, let $n'<n$, $x\in\partial_m\Omega\setminus N$ and suppose that $u_{n',k}(x)>u_{n,k}(x)$. Since $(u_{n,k})\res \Omega$ is $\nu$-a.e. nondecreasing in $n$, by the absolute continuity of $m_x$ with respect to $\nu$, we have that $(u_{n,k})\res \Omega$ is $m_x$-a.e. nondecreasing in $n$, therefore
\begin{align*}
0 & <-\int_{\Omega}\left(\a_p(x,y, u_{n',k}(y)-u_{n',k}(x))-\a_p(x,y, u_{n,k}(y)-u_{n,k}(x))\right) dm_x(y)\\
 & =\frac{1}{n}|u_{n,k}(x)|^{p-2}u_{n,k}^+(x)- \frac{1}{n'}|u_{n',k}(x)|^{p-2}u_{n',k}^+(x) \\
 & \quad +\frac{1}{k}\left(|u_{n',k}(x)|^{p-2}u_{n',k}^-(x)
- |u_{n,k}(x)|^{p-2}u_{n,k}^-(x)\right)\le 0
\end{align*}
which is a contradiction. Consequently, $u_{n,k}\res \partial_m\Omega$ is $\nu$-a.e. nondecreasing in $n$. Similarly, $u_{n,k}\res \partial_m\Omega$ is $\nu$-a.e. nonincreasing in $k$.

 Then, for $\nu$-a.e. $x\in\Omega_m$, we can pass to the limit in $n$, and then in $k$, in~\eqref{sab1050a} and \eqref{sab1050}, to get $u\in L^\infty(\Omega_m,\nu)$ such that
$$\  u(x)-\int_{\Omega_m}\a_p(x,y,u(y)-u(x))dm_x(y) =\phi(x),\quad x\in\Omega,
 $$
 and
$$-\int_{\Omega}\a_p(x,y,u(y)-u(x))dm_x(y)
=\varphi(x),\quad x\in \partial_m\Omega.$$
Therefore, for $\phi\in L^\infty(\Omega,\nu)$ the range condition holds.

\noindent {\bf Step 2}.  Let us now take $\phi\in  L^{p'}(\Omega,\nu)$. Let $\phi_{n,k} := \sup\{
\inf\{\phi,n\},-k\}$,   which is nondecreasing in $n$ and nonincreasing in $k$. By Step 1, there exists   a solution $u_{n,k}\in L^\infty(\Omega_m,\nu)$ of
$$ u_{n,k}+ A^m_{\a_p,\varphi}(u_{n,k})\ni\phi_{n,k},
$$
that is,
$$  u_{n,k}(x)-\int_{\Omega_m}\a_p(x,y,u_{n,k}(y)-u_{n,k}(x))dm_x(y) =\phi_{n,k}(x),\quad x\in\Omega,
$$
 and
$$-\int_{\Omega}\a_p(x,y,u_{n,k}(y)-u_{n,k}(x))dm_x(y)
=\varphi(x),\quad x\in \partial_m\Omega.
$$

Let us see the monotonicity properties of $u_{n,k}$ . By the complete accretivity, we have that
\begin{equation}\label{dom1919}\Vert  \left(u_{n',k'}-u_{n,k}\right)^\pm  \Vert_{L^p(\Omega,\nu)}\leq  \Vert \left(\phi_{n',k'}-\phi_{n,k}\right)^\pm  \Vert_{L^p(\Omega,\nu)}
\end{equation}
and
$$\Vert  \left(u_{n',k'}-u_{n,k}\right)^\pm  \Vert_{L^{p'}(\Omega,\nu)}\leq  \Vert \left(\phi_{n',k'}-\phi_{n,k}\right)^\pm  \Vert_{L^{p'}(\Omega,\nu)} \, .
$$
This implies, for example, that if $n'<n$ then $u_{n',k}\le u_{n,k}$ $\nu$-a.e. in $\Omega$ thus, as before, $u_{n,k}\res\Omega$ is $\nu$-a.e. nondecreasing in $n$ and $\nu$-a.e. nonincreasing in $k$. Moreover, it also implies the convergence of $u_{n,k}\res\Omega$ in $L^p(\Omega,\nu)$.

 On the other hand, for $n'<n$, we have
$$ \int_{\Omega}\left(\a_p(x,y, u_{n',k}(y)-u_{n',k}(x))-\a_p(x,y, u_{n,k}(y)-u_{n,k}(x))\right) dm_x(y)=0 \  $$
for every $x\in \partial_m\Omega$ and, therefore, the same reasoning as before yields that $u_{n,k}(x)$ is nondecreasing in $n$ and nonincreasing in $k$ for $\nu$-a.e. $x\in\partial_m \Omega$.

We want to pass to the limit in \label{dom2001}
\begin{equation}\label{sab1927}
 \left\{\begin{array}{lr}
\displaystyle u_{n,k}(x)-\int_{\Omega_m}\a_p(x,y,  u_{n,k}(y) -u_{n,k}(x))dm_x(y) =\phi_{n,k}(x),\quad x\in\Omega, & \qquad{\rm (a)}
 \\ \\
 \displaystyle
 -\int_{\Omega}\a_p(x,y, u_{n,k}(y)-  u_{n,k}(x) )dm_x(y)
=\varphi(x),\quad x\in \partial_m\Omega. & \qquad{\rm (b)}
\end{array}\right.
\end{equation}

We start by letting $n\to +\infty$. By~\eqref{dom1919}, we have that $u_{n,k}\to u_k$ in $L^p(\Omega,\nu)$. Hence, there exists $ h_k\in L^{p}(\Omega,\nu)$ such that
$$|u_{n,k}|\le h_k \ \hbox { $\nu$-a.e. in  } \Omega.$$
Note that $h_k\in L^{p}(\Omega,m_x)$ for $\nu$-a.e. $x\in \partial_m\Omega$, let $B\subset X$ be the $\nu$-null set where this is not satisfied and such that $\varphi(x)<+\infty$ for $x\in\partial_m\Omega\setminus B$.

Suppose that there exists $x\in\partial_m\Omega\setminus B$ such that $u_{n,k}(x)\to +\infty$. Then, given $M>0$, there exists $n_0$ such that, for $n\ge n_0$, $u_{n,k}(x)>M$. Hence, for $n\ge n_0$,
$$-\a_p(x,., u_{n,k}(.)-  u_{n,k}(x) )\ge -\a_p(x,.,h_k(.)-M) \in L^{p'}(\Omega, m_x),$$
so we may apply Fatou's lemma to obtain:
$$\int_\Omega\liminf_n -\a_p(x,y, u_{n,k}(y)-  u_{n,k}(x) )dm_x(y) \qquad \qquad$$
$$\qquad \qquad \le\liminf_n \int_\Omega -\a_p(x,y, u_{n,k}(y)-  u_{n,k}(x) )dm_x(y)=\varphi(x).$$
However, this is a contradiction since $u_{n,k}(x)\to +\infty$, $\varphi(x)<+\infty$ and $\lim_nu_{n,k}(y)=u_k(y) <+\infty$ for $m_x$-a.e $y\in\Omega$.

Therefore, for $\nu$-almost every $x\in\partial_m\Omega$, $u_{n,k}(x)\to u_k(x)<+\infty$ (thus, in particular, $u_k$ is $\nu$-measurable on $\partial_m\Omega$) and we can use the dominated convergence theorem to pass to the limit in~\eqref{sab1927}(b) for $\nu$-a.e. $x\in\partial_m\Omega$, obtaining:
$$ -\int_{\Omega}\a_p(x,y, u_k (y)-  u_k (x) )dm_x(y)
=\varphi(x),\quad x\in\partial_m\Omega.$$
Indeed, note that
$$|\a_p(x,.,u_{n,k}(.)-u_{n,k}(x))|\le \tilde{C}\left(1+\max\{|u_{1,k}(x)|,|u_{k}(x)|\}^{p-1}+|h_k(.)|^{p-1}\right)\in L^{p'}(\Omega,\nu)$$
for $\nu$-a.e. $x\in\partial_m\Omega$, and
$$|\a_p(x,.,u_{n,k}(.)-u_{n,k}(x))|\le \tilde{C}\left(1+\max\{|u_{1,k}(x)|,|u_{k}(x)|\}^{p-1}+|h_k(.)|^{p-1}\right)\in L^{p'}(\Omega,m_x)$$
for $\nu$-a.e. $x\in\partial_m\Omega$. Consequently, we also obtain that
$$\a_p(x,.,u_{k}(.)-u_{k}(x))\in L^{p'}(\Omega,\nu )\cap L^{p'}(\Omega,m_x )$$
for $\nu$-a.e. $x\in\partial_m\Omega$.

Now,
$$\int_{\Omega_m}\a_p(x,y,  u_{n,k}(y) -u_{n,k}(x))dm_x(y) =u_{n,k}(x)-\phi_{n,k}(x),\quad x\in\Omega,$$
thus, by the monotonicity of $\{u_{n,k}\}_n$,
$$\int_{\Omega_m}\a_p(x,y,  u_{n,k}(y) -u_{k}(x))dm_x(y) \le u_{n,k}(x)-\phi_{n,k}(x),\quad x\in\Omega,$$
and
$$ \int_{\Omega_m}\a_p(x,y,  u_{k}(y) -u_{n,k}(x))dm_x(y) \ge u_{n,k}(x)-\phi_{n,k}(x),\quad x\in\Omega .$$
Now, the right hand sides converge for $\nu$-a.e. $x\in\Omega$ and, for $\nu$-a.e. $x\in\Omega$, $\{\a_p(x,.,  u_{n,k}(.) -u_{k}(x))\}_n$ is a nondecreasing sequence   with bounded $m_x$-integrals and
$$\a_p(x,.,  u_{n,k}(.) -u_{k}(x))\ge \a_p(x,.,  u_{1,k}(.) -u_{k}(x))\in L^\infty(\Omega_m,\nu) ,$$
and $\{\a_p(x,.,  u_{k}(.) -u_{n,k}(x))\}_n$ is a nonincreasing sequence   with bounded $m_x$-integrals and
$$\a_p(x,.,  u_{k}(.) -u_{n,k}(x))\le \a_p(x,.,  u_{1}(.) -u_{n,k}(x))\in L^\infty(\Omega_m,\nu) ,$$
so we may apply the monotone convergence theorem to get
$$\int_{\Omega_m}\a_p(x,y,  u_{n,k}(y) -u_{k}(x))dm_x(y) \to \int_{\Omega_m}\a_p(x,y,  u_{k}(y) -u_{k}(x))dm_x(y),\ \hbox{ for $\nu$-a.e. } x\in\Omega \ ,$$
and
$$ \int_{\Omega_m}\a_p(x,y,  u_{k}(y) -u_{n,k}(x))dm_x(y) \to \int_{\Omega_m}\a_p(x,y,  u_{k}(y) -u_{k}(x))dm_x(y),\ \hbox{ for $\nu$-a.e. } x\in\Omega \ ,$$
obtaining also that $a_p(x,\cdot,u_k(\cdot)-u_k(x))\in L^1(\Omega_m, m_x)$ for $\nu$-a.e. $x\in\Omega$.
\bigskip
Consequently, since
$$\int_{\Omega_m}\a_p(x,y,  u_{n,k}(y) -u_{k}(x))dm_x(y)\qquad\qquad\qquad\qquad $$
$$\le\int_{\Omega_m}\a_p(x,y,  u_{n,k}(y) -u_{n,k}(x))dm_x(y) $$
$$\qquad\qquad\qquad\qquad \le\int_{\Omega_m}\a_p(x,y,  u_{k}(y) -u_{n,k}(x))dm_x(y), $$
we can also pass  to the limit in \eqref{sab1927}(b)  to get
$$u_k (x)-\int_{\Omega_m}\a_p(x,y,  u_k (y) -u_k(x))dm_x(y) =\phi_k(x),\quad x\in\Omega.$$

  The same argument but integrating over $\Omega\times\Omega_m$ with respect to $\nu\otimes m_x$ gives that $a_p(x,y,u_k(y)-u_k(x))\in L^1(\Omega\times\Omega_m, \nu\otimes m_x)$. Moreover, the same reasoning on
$$-\int_{\Omega}\a_p(x,y,  u_{n,k}(y) -u_{n,k}(x))dm_x(y) =\varphi(x),\quad x\in\partial_m\Omega$$
gives that $a_p(x,y,u_k(y)-u_k(x))\in L^1(\partial_m\Omega\times\Omega,\nu\otimes m_x)$, so we get that $a_p(x,y,u_k(y)-u_k(x))\in L^1(Q_2,\nu\otimes m_x)$.

 Finally, we take limits as $k\to +\infty$.
We may repeat the previous reasoning to obtain that $u_k(x)\to u(x)>-\infty$ for $\nu$-a.e. $x\in\partial_m\Omega$. Consequently, we have that $u_k\to u$ in $L^p(\Omega,\nu)$ and $u_k$ tends to a measurable $\nu$-a.e. finite function $u$ in $\partial_m\Omega$. Then, we apply the monotone convergence theorem in the same way to get:
\begin{equation}\label{ind002}
\left\{\begin{array}{l}
\displaystyle u (x)-\int_{\Omega_m}\a_p(x,y,  u (y) -u(x))dm_x(y) =\phi(x),\quad x\in\Omega\\ \\
\displaystyle -\int_{\Omega}\a_p(x,y, u (y)-  u(x) )dm_x(y)
=\varphi(x),\quad x\in\partial_m\Omega,
\end{array}\right.
\end{equation}
  where $$\a_p(x,y,u(y)-u(x))\in L^1(Q_2,\nu\otimes m_x).$$

By \eqref{llo2}, we have that
$$c|u(y)-u(x)|^{p-1}\leq |\a_p(x,y,u(y)-u(x))|,$$
thus
$$|u(y)|^{p-1}\leq \tilde{C}\left(|\a_p(x,y,u(y)-u(x))|+|u(x)|^{p-1}\right)  $$
for every $x$, $y\in\Omega_m$ and some constant $\tilde{C}$. Therefore, since $m_x(\Omega_m)=1$ for $x\in\Omega$,
$$\int_{\Omega_m}m_x(\Omega)|u(x)|^{p-1}d\nu(x)=\int_{\Omega_m}\int_{\Omega}|u(x)|^{p-1}dm_x(y)d\nu(x)$$
$$=\int_{\Omega}\int_{\Omega_m}|u(y)|^{p-1}dm_x(y)d\nu(x)$$
$$\leq \tilde{C}\left(\int_{Q_2}|\a_p(x,y,u(y)-u(x))|d(\nu\otimes m_x)(x,y)+\int_{\Omega}|u(x)|^{p-1}d\nu(x)\right)<+\infty .$$
This implies, in particular, that
\begin{equation}\label{ind003mmmm} m_{(\cdot)}(\Omega)|u|^{p-1}\in L^1(\partial_m\Omega,\nu).
\end{equation}
\end{proof}

\begin{remark}[\bf Regularity for $p\ge 2$]\label{regul001} In the context of Theorem~\ref{CAandRangeCond}, let us see that, for $p\ge 2$,
\begin{equation}\label{ind001}\a_p(x,y,u(y)-u(x))\in L^{p'}(Q_2,\nu\otimes m_x)
\end{equation}
and
\begin{equation}\label{ind003pmay2} m_{(\cdot)}(\Omega)|u|^{p}\in L^1(\partial_m\Omega,\nu).
\end{equation}
 Indeed, by~\eqref{ind003mmmm},  since $0\le m_{(\cdot)}(\Omega)\le 1$,
\begin{equation}\label{ind003also} m_{(\cdot)}^{p-1}(\Omega)|u|^{p-1}\in L^1(\partial_m\Omega,\nu).
\end{equation}
Therefore,
  $$m_{(\cdot)}(\Omega)u\in L^{p-1}(\partial_m\Omega,\nu)\subset L^{1}(\partial_m\Omega,\nu).$$
  Hence,  we get that
$$u\varphi=m_{(\cdot)}(\Omega)u\frac{\varphi}{m_{(\cdot)}(\Omega)}\in L^{1}(\partial_m\Omega
,\nu)  . $$
Now, multiplying the first equation in~\eqref{ind002} by $T_ku(x)$, integrating over $\Omega$ and then integrating by parts,
$$\int_\Omega uT_kud\nu+\frac12\int_{Q_2}\a_p(x,y,u(y)-u(x))\big(T_ku(y)-T_ku(x)\big)d(\nu\otimes m_x)(x,y)$$ $$=\int_\Omega\phi T_kud\nu+\int_{\partial_m\Omega}\varphi T_kud\nu.$$
Hence, letting $k\to\infty$, by Fatou's lemma,
$$(x,y)\mapsto \a_p(x,y,u(y)-u(x))\big( u(y)- u(x)\big)\in L^1(Q_2,\nu\otimes m_x),$$
and this is equivalent, on account of~\eqref{llo1} and~\eqref{llo2}, to~\eqref{ind001}.   Moreover, in this situation, we can repeat the argument used to obtain~\eqref{ind003mmmm} but using $p$ instead of $p-1$, to get~\eqref{ind003pmay2}.
 Indeed,
$$\int_{\Omega_m}m_x(\Omega)|u(x)|^{p}d\nu(x)=\int_{\Omega_m}\int_{\Omega}|u(x)|^{p}dm_x(y)d\nu(x)$$
$$=\int_{\Omega}\int_{\Omega_m}|u(y)|^{p}dm_x(y)d\nu(x)$$
$$\leq \tilde{C}\left(\int_{Q_2}\a_p(x,y,u(y)-u(x))(u(y)-u(x))d(\nu\otimes m_x)(x,y)+\int_{\Omega}|u(x)|^{p}d\nu(x)\right)<+\infty .$$
\end{remark}

  With respect to the domain of the operator $A^m_{\a_p,\varphi}$, we have the following result.

\begin{theorem}\label{Nremdom01ferrogen} Let $\varphi\in L^{m,\infty}(\partial_m\Omega,\nu)$. Then, we have
 $$L^{\infty}(\Omega,\nu)\subset D(A^m_{\a_p,\varphi})$$
 and, consequently,
  $$\overline{D(A^m_{\a_p,\varphi})}^{L^{p'}(\Omega,\nu)}=L^{p'}(\Omega,\nu).$$
\end{theorem}

\begin{proof}
Take $u\in L^\infty(\Omega,\nu)$.  By Remark~\ref{elrem001}.{\it 1}~\&~{\it 3}, there exists an extension of $u$ to $\partial_m\Omega$ (which we continue to denote by $u$) satisfying
$$
-\int_\Omega \a_p(x,y,u(y)-u(x))dm_x(y)=\varphi(x), \quad x\in\partial_m\Omega,
$$
and, moreover,
  $$ u\in L^\infty(\partial_m\Omega,\nu).$$
   Therefore, for $x\in\Omega$,
 $$\phi(x)=-\int_{\Omega_m}\a_p(x,y,u(y)-u(x))dm_x(y)$$
 defines a function in $L^1(\Omega,\nu)$, and we have that
 $$u\in D(A^m_{\a_p,\varphi}).$$
\end{proof}

\begin{theorem}\label{forlin001gen1}   Let $p\ge 2$ and assume that   $\varphi\in L^{m,\infty}(\partial_m\Omega,\nu)$. Then
  $$L^{p-1}(\Omega,\nu)\subset D(A^m_{\a_p,\varphi}).$$

\end{theorem}
\begin{proof}
Suppose that $p>2$ ( the case $p=2$ follows by a similar, but simpler, argument).

Given $u\in L^{p-1}(\Omega,\nu)$,  denote again by $u$ the unique extension of $u$ to the boundary $\partial_m\Omega$ satisfying
\begin{equation}\label{rees001gen1}
-\int_{\Omega}\a_p(x,y,u(y)-u(x))dm_x(y)=\varphi(x), \ x\in\partial_m\Omega.
\end{equation}
Then, for $x \in \partial_m \Omega$, we have
\begin{equation}\label{ide001}
\begin{array}{l}
\displaystyle u(x)\int_{\{y\in\Omega:|u(y)-u(x)|>1\}}
\frac{\a_p(x,y,u(y)-u(x))}{u(y)-u(x)}dm_x(y)
\\ \\
\displaystyle
=\varphi(x)+\int_{\{y\in\Omega:|u(y)-u(x)|\le 1\}}
 \a_p(x,y,u(y)-u(x))dm_x(y)
 \\ \\
\displaystyle
\quad +\int_{\{y\in\Omega:|u(y)-u(x)|>1\}}\frac{\a_p(x,y,u(y)-u(x))}{u(y)-u(x)}u(y)dm_x(y),
\end{array}
\end{equation}
and, consequently, by~\eqref{llo1} and taking into account ~\eqref{llo2},
\begin{equation}\label{ide001modu}
\begin{array}{l}
\displaystyle |u(x)|\int_{\{y\in\Omega:|u(y)-u(x)|>1\}}
\frac{ \a_p(x,y,u(y)-u(x)) }{ u(y)-u(x) }dm_x(y)
\\ \\
\displaystyle
\le |\varphi(x)|+2C
+2C\int_{\{y\in\Omega:|u(y)-u(x)|>1\}}
|u(y)-u(x)|^{p-2}|u(y)|dm_x(y) .
\end{array}
\end{equation}
Now, by~\eqref{llo2},
$$c\int_{\{y\in\Omega:|u(y)-u(x)|>1\}}|u(y)-u(x)|^{p-2}dm_x(y)\qquad\qquad$$ $$\qquad\qquad\le
\int_{\{y\in\Omega:|u(y)-u(x)|>1\}}\frac{\a_p(x,y,u(y)-u(x))}{ u(y)-u(x)}dm_x(y).$$
Hence, by~\eqref{ide001modu}, we get
\begin{equation}\label{bieen}\begin{array}{c}
\displaystyle
c|u(x)|\int_{\{y\in\Omega:|u(y)-u(x)|>1\}}
|u(y)-u(x)|^{p-2}dm_x(y)
\qquad\qquad\qquad
\\ \\ \displaystyle
\le|\varphi(x)|+2C
+2C\int_{\{y\in\Omega:|u(y)-u(x)|>1\}}|u(y)-u(x)|^{p-2}|u(y)|dm_x(y).
\end{array}
\end{equation}

Let us now see that
\begin{equation}\label{domh002gene}
\mathbf{\Theta}:= \int_{\partial_m\Omega}\int_\Omega|u(y)-u(x)|^{p-1}dm_x(y)d\nu(x)<+\infty.
\end{equation} By~\eqref{bieen}   and the reversibility of  $\nu$, we have
$$
\begin{array}{l}
\displaystyle
 \int_{\partial_m\Omega}\int_\Omega|u(y)-u(x)|^{p-1}dm_x(y)d\nu(x)
\\ \\
\displaystyle
= \int_{\partial_m\Omega}\int_{\{y\in\Omega:|u(y)-u(x)|\le1\}} |u(y)-u(x)|^{p-1}dm_x(y)d\nu(x)
\\ \\
\displaystyle
\qquad+ \int_{\partial_m\Omega}\int_{\{y\in\Omega:|u(y)-u(x)|>1\}}|u(y)-u(x)|^{p-1}dm_x(y)d\nu(x)
\\ \\
\displaystyle\le  \nu(\Omega_m)+ \int_{\partial_m\Omega}\int_{\{y\in\Omega:|u(y)-u(x)|>1\}}|u(y)-u(x)|^{p-2}|u(y)|dm_x(y)d\nu(x)
\\ \\
\displaystyle
\qquad+
\int_{\partial_m\Omega}|u(x)|\int_{\{y\in\Omega:|u(y)-u(x)|>1\}}|u(y)-u(x)|^{p-2} dm_x(y)\nu(x)\\
\\
\displaystyle
 \le \left(1+2\frac{C}{c}\right)\nu(\Omega_m)+ \frac{1}{c}\int_{\partial_m\Omega}|\varphi(x)|d\nu(x)
 \\
\\
\displaystyle
\qquad + \left(1+2\frac{C}{c}\right)\int_{\partial_m\Omega}\int_\Omega|u(y)-u(x)|^{p-2}|u(y)|dm_x(y)d\nu(x) .
 \end{array}
$$
Now, by using H\"{o}lder's inequality, with exponents $\frac{p-1}{p-2}$ and $p-1$, and the reversibility of  $\nu$, we get
$$
\begin{array}{l}
\displaystyle
 \int_{\partial_m\Omega}\int_\Omega|u(y)-u(x)|^{p-2}|u(y)|dm_x(y)d\nu(x)
\\ \\
\displaystyle \le
\left(\int_{\partial_m\Omega}\int_\Omega|u(y)-u(x)|^{p-1} dm_x(y)d\nu(x)\right)^{\frac{p-2}{p-1}}
\left(\int_{\partial_m\Omega}\int_\Omega|u(y)|^{p-1} dm_x(y)d\nu(x)\right)^{\frac{1}{p-1}}
\\ \\
\displaystyle \le
\left(\int_{\partial_m\Omega}\int_\Omega|u(y)-u(x)|^{p-1} dm_x(y)d\nu(x)\right)^{\frac{p-2}{p-1}}
\left( \int_\Omega|u(x)|^{p-1} d\nu(x)\right)^{\frac{1}{p-1}}.
\end{array}
$$
Therefore,
$$\mathbf{\Theta}\le \left(1+2\frac{C}{c}\right)\nu(\Omega_m)+ \frac{1}{c}\int_{\partial_m\Omega}|\varphi(x)|d\nu(x)+
 \left(1+2\frac{C}{c}\right)\mathbf{\Theta}^{\frac{p-2}{p-1}}\left( \int_\Omega|u(x)|^{p-1} d\nu(x)\right)^{\frac{1}{p-1}}$$
and, consequently,
$\mathbf{\Theta}$ is finite.
Observe that an explicit upper bound, depending on $\Vert \varphi \Vert_{L^{1}(\Omega, \nu)}$ and $\Vert u\Vert_{L^{p-1}(\Omega, \nu)}$, can be stated.

\medskip

Furthermore, we obtain the following regularity of $u$ on the boundary:
\begin{equation}\label{nolv01gene}
\int_{\partial_m\Omega}m_x(\Omega)|u(x)|^{p-1}d\nu(x)<+\infty.
\end{equation}
Indeed, since
$|u(y)|^{p-1}\leq \tilde{C}\left(|u(y)-u(x)|^{p-1}+|u(x)|^{p-1}\right)  $
for some constant $\tilde{C}$ and every $x$, $y\in\Omega_m$, we have that
$$\int_{\Omega_m}m_x(\Omega)|u(x)|^{p-1}d\nu(x)=\int_{\Omega_m}\int_{\Omega}|u(x)|^{p-1}dm_x(y)d\nu(x)$$
$$=\int_{\Omega}\int_{\Omega_m}|u(y)|^{p-1}dm_x(y)d\nu(x)$$
$$\leq \tilde{C}\left(\int_{Q_2}|u(y)-u(x)|^{p-1}d(\nu\otimes m_x)(x,y)+\int_{\Omega}|u(x)|^{p-1}d\nu(x)\right),$$
thus \eqref{nolv01gene} holds.

\medskip
Let us finally see that, for $x\in\Omega$,
 $$\phi(x):=-\int_{\Omega_m}\a_p(x,y,u(y)-u(x))dm_x(y)$$
belongs to $L^1(\Omega,\nu)$. Indeed,
 $$\begin{array}{c}
 \displaystyle-\int_{\Omega_m}\a_p(x,y,u(y)-u(x))dm_x(y)
 \\ \\ \displaystyle
 =-\int_{\Omega}\a_p(x,y,u(y)-u(x))dm_x(y)
 -\int_{\partial_m\Omega}\a_p(x,y,u(y)-u(x))dm_x(y).
 \end{array}
 $$
 Now, the first summand on the right hand side belongs to $L^1(\Omega,\nu)$. Let us see that the second one also belongs to $L^1(\Omega,\nu)$. Since
$$
\begin{array}{l}\displaystyle
\int_\Omega\left|\int_{\partial_m\Omega}\a_p(x,y,u(y)-u(x))dm_x(y)\right|d\nu(x)
\\ \\ \displaystyle
\le \int_\Omega\int_{\partial_m\Omega}|\a_p(x,y,u(y)-u(x))|dm_x(y)d\nu(x)
\\ \\ \displaystyle
\le \int_\Omega\int_{\partial_m\Omega}C(1+|u(y)-u(x)|^{p-1})dm_x(y) d\nu(x)
\\ \\ \displaystyle
\le C\nu(\Omega_m)+ C\int_\Omega\int_{\partial_m\Omega} |u(y)-u(x)|^{p-2}|u(y)|dm_x(y) d\nu(x)
\\ \\ \displaystyle
\quad+  C\int_\Omega\int_{\partial_m\Omega} |u(y)-u(x)|^{p-2} dm_x(y) |u(x)|d\nu(x),
\end{array}
$$
 we have that,   by Tonelli-Hobson's theorem, $\displaystyle x\mapsto -\int_{\partial_m\Omega}\a_p(x,y,u(y)-u(x))dm_x(y)$ belongs to $L^1(\Omega,\nu)$
  if the following functions belong to $L^1(\Omega,\nu)$:
 $$x\longmapsto \int_{\partial_m\Omega}|u(y)-u(x)|^{p-2}u(y)dm_x(y)$$ and
 $$x\longmapsto\int_{\partial_m\Omega}|u(y)-u(x)|^{p-2} dm_x(y)u(x).$$
With regard to the first function, by H\"{o}lder's inequality and the reversibility of  $\nu$ with respect to $m$, we have that
$$\begin{array}{l}
\displaystyle
\int_\Omega\left|\int_{\partial_m\Omega}|u(y)-u(x)|^{p-2}u(y)dm_x(y)\right|d\nu(x)
\\ \\
\displaystyle
\le \int_\Omega \int_{\partial_m\Omega}|u(y)-u(x)|^{p-2}|u(y)|dm_x(y) d\nu(x)
\\ \\
\displaystyle
\le
\left(\int_\Omega\int_{\partial_m\Omega}|u(y)-u(x)|^{p-1} dm_x(y)d\nu(x)\right)^{\frac{p-2}{p-1}}
\left(\int_\Omega\int_{\partial_m\Omega}|u(y)|^{p-1} dm_x(y)d\nu(x)\right)^{\frac{1}{p-1}}
\\ \\
\displaystyle
=
\left(\int_{\partial_m\Omega}\int_\Omega|u(y)-u(x)|^{p-1} dm_x(y)d\nu(x)\right)^{\frac{p-2}{p-1}}
\left(\int_{\partial_m\Omega}\int_\Omega|u(x)|^{p-1} dm_x(y)d\nu(x)\right)^{\frac{1}{p-1}}
\\ \\
\displaystyle
=
\left(\int_{\partial_m\Omega}\int_\Omega|u(y)-u(x)|^{p-1} dm_x(y)d\nu(x)\right)^{\frac{p-2}{p-1}}
\left(\int_{\partial_m\Omega}m_x(\Omega)|u(x)|^{p-1} d\nu(x)\right)^{\frac{1}{p-1}}
\end{array}
$$
 which  is finite by~\eqref{domh002gene} and~\eqref{nolv01gene}. The second one also belongs to $L^1(\Omega,\nu)$ since,  by~\eqref{domh002gene} (using the reversibility of $\nu$ with respect to $m$),
 $\displaystyle x\mapsto \int_{\partial_m\Omega}|u(y)-u(x)|^{p-2} dm_x(y)\in L^{(p-1)'}(\Omega,\nu),$ and  $u\in L^{p-1}(\Omega,\nu).$
\end{proof}

The following theorem is a consequence of the above results thanks to Theorem~\ref{teointronls}.

\begin{theorem}\label{nsth01N} Let $\varphi\in L^{m,\infty}(\partial_m\Omega,\nu)$ and  $T>0$. For any
$u_0\in \overline{D(A^m_{\a_p,\varphi})}^{L^{p'}(\Omega,\nu)}=L^{p'}(\Omega,\nu)$  there exists a unique mild-solution $ u(t,x)$ of  Problem~\eqref{NNeumann4}.
 Moreover, for any $ q\geq p'$ and  $u_{0i}\in L^q(\Omega,\nu)$,  $i=1,2$, we have the following contraction principle for the corresponding mild-solutions $u_i$:
$$  \Vert (u_1(t,.)-u_2(t,.))^+\Vert_{L^q(\Omega,\nu)}\le \Vert (u_{0,1}-u_{0,2})^+\Vert_{L^q(\Omega,\nu)}
 \quad \hbox{for any \  $0\le t< T$.}$$

 If $u_0\in D(A^m_{\a_p,\varphi})$, then the mild-solution is a strong solution.  In particular, if $u_0\in L^{\infty}(\Omega,\nu)$,   Problem~\eqref{NNeumann4} has a unique strong solution. For $p\ge 2$ this is true for data in $L^{p-1}(\Omega,\nu)$.
\end{theorem}

\section{ Particular cases}\label{PCases}
  This section deals with the case that $\a_p$ is positive homogeneous and with two important examples of  metric random walk spaces, for which, applying the above general results, we get   existence and uniqueness of strong solutions.

\subsection{The homogeneous   Neumann  boundary value problem}

  \begin{definition}{\rm   We will say that $\a_p$ is {\it positive homogeneous} if
$$ \a_p(x,y, \lambda r) = \lambda^{p-1} \a_p(x,y, r) \ \hbox{ for every }   \lambda>0, \  x,y \in X \hbox{ and }  r \in \R.$$
}
\end{definition}
For example, if $$\a_p(x,y,r)=\frac{\varphi(x)+\varphi(y)}{2}|r|^{p-2}r,$$
 where $\varphi:X\rightarrow \R$ is a bounded $\nu$-measurable function, then $\a_p$ is positive
homogeneous.

It follows that, if $\a_p$ is  positive homogeneous  then the operator  $B^m_{\a_p,0}$ is
 positive homogeneous of degree $p-1$, that is, $B^m_{\a_p,0}(\lambda u) = \lambda^{p-1} B^m_{\a_p,0}( u)$
 for every $u \in D(B^m_{\a_p,0})$ and $\lambda>0$. Then, since $B^m_{\a_p,0}$ is an $m$-completely accretive operator, we have that,
 by the results in \cite{BCr2} (see Theorem \ref{teointronls}), the mild solutions  of Problem~\eqref{Neumann4} are, in fact, strong solutions  if $p-1 \not= 1$.
Consequently,  under the Assumptions in Section~\ref{aus23}, we have the following result.

\begin{theorem}\label{nsth01S} Let $ p\neq 2$ and assume that $\a_p$ is  positive homogeneous.
  For any
$u_0\in \overline{D(B^m_{\a_p,0})}^{L^{p'}(\Omega,\nu)}=L^{p'}(\Omega,\nu)$  there exists a unique strong solution $ u(t,x)$ of Problem~\eqref{Neumann4} with $\varphi =0$.
 Moreover, for any $ q\geq p'$ and $u_{0i}\in L^q(\Omega,\nu)$,  $i=1,2$, we have the following contraction principle for the corresponding strong solutions $u_i$:
$$  \Vert (u_1(t,.)-u_2(t,.))^+\Vert_{L^q(\Omega,\nu)}\le \Vert (u_{0,1}-u_{0,2})^+\Vert_{L^q(\Omega,\nu)}
 \quad \hbox{for any \  $0\le t< T$.}$$
\end{theorem}
Similarly, under the Assumptions in Section~\ref{aus24}, we can state the corresponding result for Problem~\eqref{NNeumann4}.

\medskip
Consider $p=2$, for which   the last statement in Theorem \ref{teointronls}  does not apply.

 \begin{lemma}\label{Pcontr1} Let $u_i\in L^2(\Omega,\nu)$, $i=1,2$, and assume that  there exists  $ \overline u_i\in L^2(\Omega_m,\nu)$  with $\overline{u_i}_{\vert \Omega} = u_i$ (that we denote equally as $u_i$) satisfying
\begin{equation}\label{ben001}
0= -\int_{\Omega_m} (u_i(y)-u_i(x)) dm_x(y), \quad x \in \partial_m\Omega.
\end{equation}
 Then,
 \begin{equation}\label{EPcontr1}\begin{array}{l}
 \displaystyle\int_{\partial_m \Omega} m_x(\Omega) (u_1(x) - u_2(x))^2 d \nu(x) \\ \\ \quad +  \displaystyle\int_{\partial_m \Omega\times\partial_m \Omega}  \left((u_1(y) - u_2(y)) -(u_1(x)- u_2(x)) \right)^2 dm_x(y) d\nu(x) \\ \\ \leq  \displaystyle\int_\Omega (u_1(x) - u_2(x)) d \nu(x).\end{array}
 \end{equation}
 \end{lemma}
 \begin{proof} \eqref{ben001} is equivalent to
 $$m_x(\Omega)u_i(x)   - \int_{\partial_m \Omega} (u_i(y)-u_i(x)) dm_x(y) = \int_\Omega u_i(y) dm_x(y) \quad x \in \partial_m \Omega,  $$
 for $i = 1,2.$
 Hence, for $x \in \partial_m \Omega$, we have
 $$\begin{array}{l} \displaystyle m_x(\Omega)(u_1(x) - u_2(x))  - \int_{\partial_m \Omega} \left((u_1(y) - u_2(y)) -(u_1(x)- u_2(x)) \right) dm_x(y) \\ \\ \displaystyle = \int_\Omega (u_1(y) - u_2(y)) dm_x(y).\end{array}$$
Then, multiplying by $(u_1(x) - u_2(x))$, integrating over $\partial_m \Omega$ with respect to $\nu$ and applying integration by parts and the reversibility of $\nu$ with respect to $m$, we get
$$\begin{array}{l}  \displaystyle\int_{\partial_m \Omega} m_x(\Omega)(u_1(x) - u_2(x))^2  d\nu(x) \\ \\
 \displaystyle \quad + \frac12 \int_{\partial_m \Omega} \int_{\partial_m \Omega}  \left((u_1(y) - u_2(y)) -(u_1(x)- u_2(x)) \right)^2 dm_x(y) d\nu(x)\\ \\
 \displaystyle= \int_{\partial_m \Omega} \int_\Omega (u_1(y) - u_2(y))(u_1(x) - u_2(x)) dm_x(y)d \nu(x)\\ \\
 \displaystyle\leq \frac12 \int_{\partial_m \Omega} \int_\Omega   (u_1(x) - u_2(x))^2 dm_x(y)d \nu(x) +\frac12  \int_{\partial_m \Omega} \int_\Omega  (u_1(y) - u_2(y))^2 dm_x(y)d \nu(x) \\ \\
 \displaystyle=\frac12  \int_{\partial_m \Omega} (u_1(x) - u_2(x))^2 m_x(\Omega) d\nu(x) +\frac12  \int_\Omega \int_{\partial_m \Omega} (u_1(x) - u_2(x))^2 dm_x(y)d \nu(x) \\ \\
 \displaystyle\leq \frac12  \int_{\partial_m \Omega} (u_1(x) - u_2(x))^2 m_x(\Omega) d\nu(x) +\frac12  \int_\Omega   (u_1(x) - u_2(x))^2 d \nu(x)\end{array}$$
 \end{proof}

As a consequence of the above result, given $u\in L^2(\Omega,\nu)$,  if  there exists  $T(u)\in L^2(\Omega_m,\nu)$   with $T(u)_{\vert \Omega} = u$   and satisfying
\begin{equation}\label{ole}
0= -\int_{\Omega_m} (T(u)(y)-T(u)(x)) dm_x(y), \quad x \in \partial_m\Omega,
\end{equation}
then $T(u)$ is unique.

 Let us consider the nonempty convex set
$$K:= \{ u\in L^2(\Omega,\nu) \ : \ \exists T(u)\in L^2(\Omega_m,\nu) \ \hbox{with $T(u)_{\vert \Omega} = u$ and satisfying} \ \eqref{ole} \}, $$
and the energy  operator $\mathcal {F} : L^2(\Omega,\nu) \rightarrow (-\infty, +\infty]$ given by
 \begin{equation}\label{operator1}
 \mathcal {F}(u):= \left\{\begin{array}{ll} \displaystyle \frac14 \int_{\Omega_m \times \Omega_m}\left(T(u)(y) - T(u)(x) \right)^2 d(\nu\otimes m_x)(x,y), &\quad  u \in K,\\ \\ +\infty,   &\quad\hbox{else.}  \end{array}   \right.
 \end{equation}
It follows that $\mathcal {F}$ is proper and convex. Moreover, we also have:

 \begin{lemma}\label{LSCC} The operator $\mathcal {F}$ is lower semi-continuous in $L^2(\Omega,\nu)$.
 \end{lemma}
 \begin{proof} Let $u_n \in L^2(\Omega,\nu)$ such that $u_n \to u$ in $L^2(\Omega,\nu)$. We can assume that $$\liminf_{n \to \infty} \mathcal {F}(u_n) < +\infty.$$
 Hence, without loss of generality, we can assume that $u_n \in K$ for all $n \in \N$ and
 $$\liminf_{n \to \infty} \mathcal {F}(u_n) = \lim_{n \to \infty} \mathcal {F}(u_n).$$

 Note that, by Lemma \ref{Pcontr1}, we have
 \begin{equation}\label{E1lemma}
 m_x(\Omega)^{\frac12} T(u_n) \to m_x(\Omega)^{\frac12} T(u) \quad \hbox{in} \ \ L^2(\partial_m \Omega,\nu),
 \end{equation}
 and
 \begin{equation}\label{E2lemma}
T(u_n)(y) - T(u_n)(x) \to T(u)(y) - T(u)(x) \quad \hbox{in} \ \ L^2(\partial_m \Omega \times \partial_m \Omega, \nu\otimes m_x).
 \end{equation}

Now, by the reversibility of $\nu$, we have
\begin{align*}4 \mathcal {F}(u_n) &= \int_{\Omega \times \Omega} (u_n(y) - u_n(x))^2 d(\nu\otimes m_x)(x,y)\\
  & \quad + 2 \int_{\partial_m \Omega} \int_\Omega \left(T(u_n)(y) - T(u_n)(x) \right)^2  d m_x(y)d\nu(x)  \\  & \quad +\int_{\partial_m \Omega \times \partial_m \Omega} \left(T(u_n)(y) - T(u_n)(x) \right)^2 d(\nu\otimes m_x)(x,y).
  \end{align*}
 Let's see what happens term by term. Since $u_n \to u$ in $L^2(\Omega,\nu)$ and by \eqref{E2lemma}, we have
 $$\lim_{n \to \infty} \int_{\Omega \times \Omega} (u_n(y) - u_n(x))^2 d(\nu\otimes m_x)(x,y) = \int_{\Omega \times \Omega} (u(y) - u(x))^2 d(\nu\otimes m_x)(x,y)$$
 and
 $$\lim_{n \to \infty}\int_{\partial_m \Omega \times \partial_m \Omega} \left(T(u_n)(y) - T(u_n)(x) \right)^2 d(\nu\otimes m_x)(x,y) $$ $$= \int_{\partial_m \Omega \times \partial_m \Omega} \left(T(u)(y) - T(u)(x) \right)^2 d(\nu\otimes m_x)(x,y).$$
 On the other hand,
 $$\int_{\partial_m \Omega} \int_\Omega \left(T(u_n)(y) - T(u_n)(x) \right)^2  d m_x(y)d\nu(x) =  \int_{\partial_m \Omega} \int_\Omega \left(T(u_n)(y) \right)^2  d m_x(y)d\nu(x) $$ $$- 2 \int_{\partial_m \Omega} \int_\Omega T(u_n)(y)T(u_n)(x)   d m_x(y)d\nu(x) + \int_{\partial_m \Omega} \int_\Omega \left(T(u_n)(x) \right)^2  d m_x(y)d\nu(x).$$
 By the reversibility of $\nu$, we have
 $$\begin{array}{l} \displaystyle\int_{\partial_m \Omega} \int_\Omega \left(T(u_n)(y) \right)^2  d m_x(y)d\nu(x) = \int_\Omega \int_{\partial_m \Omega} \left(T(u_n)(x) \right)^2  d m_x(y)d\nu(x) \\ \\
 = \displaystyle\int_\Omega m_x(\partial_m \Omega) \left(T(u_n)(x) \right)^2 d\nu(x) \stackrel{n}{\longrightarrow} \int_\Omega m_x(\partial_m \Omega) \left(T(u)(x) \right)^2 d\nu(x) \\ \\
 = \displaystyle\int_{\partial_m \Omega} \int_\Omega \left(T(u)(y) \right)^2  d m_x(y)d\nu(x).\end{array}$$

 Now, by \eqref{E1lemma},
 $$\begin{array}{l} \displaystyle\int_{\partial_m \Omega} \int_\Omega \left(T(u_n)(x) \right)^2  d m_x(y)d\nu(x) \\ \\
 \displaystyle= \int_{\partial_m \Omega} m_x(\Omega) \left(T(u_n)(x) \right)^2 d\nu(x) \stackrel{n}{\longrightarrow} \int_{\partial_m \Omega} m_x(\Omega) \left(T(u)(x) \right)^2 d\nu(x) \\ \\
 \displaystyle= \int_{\partial_m \Omega} \int_\Omega \left(T(u)(x) \right)^2  d m_x(y)d\nu(x).\end{array}$$
 Finally, by the reversibility of $\nu$ with respect to $m$, we have
 $$\int_{\partial_m \Omega} \int_\Omega T(u_n)(y)T(u_n)(x)   d m_x(y)d\nu(x) = \int_\Omega u_n(x) \left(\int_{\partial_m \Omega} T(u_n)(y) dm_x(y)\right) d\nu(x).$$
However, by the reversibility of $\nu$ with respect to $m$ and \eqref{E1lemma}, we have
 $$\int_\Omega \left\vert \int_{\partial_m \Omega} T(u_n)(y) dm_x(y) - \int_{\partial_m \Omega} T(u)(y) dm_x(y) \right\vert^2 d \nu(x)$$ $$ \leq \int_\Omega  \int_{\partial_m \Omega} \left\vert T(u_n)(y) - T(u)(y)\right\vert^2 dm_x(y)  d \nu(x) $$ $$= \int_{\partial_m \Omega} \int_\Omega   \left\vert T(u_n)(x) - T(u)(x)\right\vert^2 dm_x(y)  d \nu(x)$$ $$= \int_{\partial_m \Omega} m_x(\Omega)  \left\vert T(u_n)(x) - T(u)(x)\right\vert^2   d \nu(x) \stackrel{n}{\longrightarrow} 0.$$
Hence,
 $$\int_{\partial_m \Omega} T(u_n)(y) dm_x(y)  \stackrel{n}{\longrightarrow}  \int_{\partial_m \Omega} T(u)(y) dm_x(y) \quad \hbox{in} \ L^2(\Omega, \nu),$$
 and, consequently,
 $$- 2 \int_{\partial_m \Omega} \int_\Omega T(u_n)(y)T(u_n)(x)   d m_x(y)d\nu(x) \stackrel{n}{\longrightarrow} - 2 \int_{\partial_m \Omega} \int_\Omega T(u)(y)T(u)(x)   d m_x(y)d\nu(x).$$
Therefore, we have proved that
 $$\int_{\partial_m \Omega} \int_\Omega \left(T(u_n)(y) - T(u_n)(x) \right)^2  d m_x(y)d\nu(x) \stackrel{n}{\longrightarrow} \int_{\partial_m \Omega} \int_\Omega \left(T(u)(y) - T(u)(x) \right)^2  d m_x(y)d\nu(x),$$
thus
 $$\mathcal {F}(u) =\lim_{n \to \infty} \mathcal {F}(u_n).$$
 \end{proof}

 \begin{theorem}\label{case2} If $a_2(x,y,r)=r$ then $B^m_{\a_2,0} = \partial \mathcal{F}$ and, consequently, there exists a unique strong solution $ u(t,x)$ of Problem~\eqref{Neumann4} with $\varphi=0$ for any initial datum in $L^2(\Omega,\nu)$.

 \end{theorem}
 \begin{proof} Since $\mathcal{F}$ is proper, convex and lower semi-continuous, we have that $\partial \mathcal{F}$  is maximal monotone   and $\overline{\hbox{Dom}(\partial \mathcal{F})}^{L^2(\Omega,\nu)}=\overline{\hbox{Dom}(\mathcal{F})}^{L^2(\Omega,\nu)}$. Consequently, if $B^m_{\a_p,0} \subset \partial \mathcal{F}$ then
  $$B^m_{\a_p,0} =\partial \mathcal{F}.$$ Now, given $(u,v ) \in B^m_{\a_p,0}$,  there exists  a unique  $T(u)\in L^2(\Omega_m,\nu)$  with $T(u)_{\vert \Omega} = u$ and  satisfying
 \begin{equation}\label{1ok}
  -\int_{\Omega_m} (T(u)(y)-T(u)(x)) dm_x(y) = v(x),  \quad x \in \Omega,
 \end{equation}
 and
 \begin{equation}\label{2ok}
0= -\int_{\Omega_m} (T(u)(y)-T(u)(x)) dm_x(y), \quad x \in \partial_m\Omega.
 \end{equation}
Then, given $w \in L^2(\Omega,\nu)$ such that $\mathcal {F}(w) < +\infty$,  multiplying \eqref{1ok} by $T(w) - T(u)$ and integrating over $\Omega$ with respect to $\nu$, by integrating by parts, we get
$$\begin{array}{l}  \displaystyle\int_{\Omega } v(x) \Big(T(w)(x) - T(u)(x)\Big) d\nu(x) \\ \\
\displaystyle= \frac12 \int_{\Omega_m\times\Omega_m} \big(T(u)(y)-T(u)(x)\big) \left(\big(T(w)(y) - T(u)(y)\big) - \big(T(w)(x) - T(u)(x)\big) \right) d(\nu\otimes m_x)(x,y) \\ \\
 \displaystyle=\frac12 \int_{\Omega_m\times\Omega_m} \left(\big(T(u)(y)-T(u)(x)\big)\big(T(w)(y) - T(w)(x)\big) - \big(T(u)(y) - T(u)(x)\big)^2 \right) d(\nu\otimes m_x)(x,y) \\ \\
 \displaystyle\leq  \frac14 \int_{\Omega_m\times\Omega_m} \left(\big(T(w)(y) - T(w)(x)\big)^2 -\big(T(u)(y) - T(u)(x)\big)^2 \right) d(\nu\otimes m_x)(x,y) \\ \\
  \displaystyle= \mathcal{F}(w) - \mathcal{F}(u).\end{array}$$
Therefore, $(u,v) \in \partial \mathcal{F}$ as required. \end{proof}

\begin{remark}\label{pf001}
 Assume that
the following {\it Poincar\'{e} type    inequality} holds: there exists a constant $\tilde\lambda>0$ such that, for any $u \in L^2(\partial_m\Omega,\nu)$,
\begin{equation}\label{sndpine} \left\Vert  u - \frac{1}{\nu(\partial_m\Omega)} \int_{\partial_m\Omega} u d\nu \right\Vert_{L^2(\partial_m\Omega,\nu)}  \leq \tilde\lambda\left(\int_{\partial_m\Omega\times\partial_m\Omega} |u(y)-u(x)|^2 d(\nu\otimes m_x)(x,y) \right)^{\frac12},
\end{equation}
(under rather general conditions, there
are metric random walk spaces satisfying this kind of inequality, recall the comment after~\eqref{agost004}).  Using Lemma~\ref{EPcontr1},  it is easy to see that the previously defined set $K$ is closed in $L^2(\Omega, \nu)$. Hence, since  by Theorem~\ref{case2} and Theorem~\ref{remdom01} we have that $\overline{K}^{L^2(\Omega,\nu)}=L^2(\Omega,\nu)$, we conclude that, in fact, $K=L^2(\Omega,\nu)$.

\end{remark}

\subsection{Nonlocal problems with nonsingular kernels}

Let $\Omega \subset \R^N$ be an open bounded set and  $J:\R^N\rightarrow [0,+\infty[$  a measurable, nonnegative and radially symmetric
function  verifying $\displaystyle \int J=1$. Consider the metric random walk space $[\R^N, d, m^J]$ as specified in Example \ref{graphs101}. Then, $[\R^N, d, m^J, \mathcal{L}^N]$ satisfies the Poincar\'{e}'s inequality~\eqref{agost004} (see \cite{MST2}, note that slight modifications in the results given there are required to prove our statement).
Let $\a_p(x,y,r)=|r|^{p-2}r,$ which is  positive homogeneous.
Then, if we consider the problem
\begin{equation}\label{Neumann4N}
\left\{ \begin{array}{ll} u_t(t,x) = \displaystyle\int_{\Omega_{m^J}} J(y-x) \vert u(y)-u(x)\vert^{p-2}(u(y) - u(x)) dy, \quad
   &x\in  \Omega,\ 0<t<T, \\ \\ - \displaystyle\int_{\Omega_{m^J}} J(y-x) \vert u(y)-u(x)\vert^{p-2}(u(y) - u(x)) dy = 0, \quad   &x\in\partial_{m^J}\Omega, \  0<t<T, \\ \\ u(x,0) = u_0(x),    &x\in \Omega, \end{array} \right.
\end{equation}
we can apply Theorem \ref{nsth01S} and Theorem \ref{case2} to get the following existence and uniqueness result.
\begin{theorem}\label{nsth01J}     For any
$u_0\in L^{p'}(\Omega, \mathcal{L}^N)$  there exists a unique strong solution $ u(t,x)$ of  Problem~\eqref{Neumann4N}.
 Moreover, for any $q\geq p'$ and  $u_{0i}\in L^q(\Omega,\mathcal{L}^N)$,  $i=1,2$, we have the following contraction principle for the corresponding strong solutions $u_i$:
$$  \Vert (u_1(t,.)-u_2(t,.))^+\Vert_{L^q(\Omega,\mathcal{L}^N)}\le \Vert (u_{0,1}-u_{0,2})^+\Vert_{L^q(\Omega,\mathcal{L}^N)}
 \quad \hbox{for any \  \ $0\le t< T$.}$$
\end{theorem}

Consider now the problem
\begin{equation}\label{Neumann4NValdi}
\left\{ \begin{array}{ll} u_t(t,x) = \displaystyle\int_{\Omega_{m^J}} J(y-x) \vert u(y)-u(x)\vert^{p-2}(u(y) - u(x)) dy, \quad
   &x\in  \Omega,\ 0<t<T, \\ \\ - \displaystyle\int_{\Omega} J(y-x) \vert u(y)-u(x)\vert^{p-2}(u(y) - u(x)) dy = \varphi(x), \quad   &x\in\partial_{m^J}\Omega, \  0<t<T, \\ \\ u(x,0) = u_0(x),    &x\in \Omega. \end{array} \right.
\end{equation}
Applying Theorem \ref{nsth01N}, we get the following existence and uniqueness result.

\begin{theorem}\label{euValdi}
 Let $\varphi\in L^{m,\infty}(\partial_m\Omega,\mathcal{L}^N)$. For every $u_0 \in L^\infty (\Omega,\mathcal{L}^N)$ there exists a unique strong solution of Problem \eqref{Neumann4NValdi}. If $p\ge 2$, this is also true for data in $L^{p-1}(\Omega,\mathcal{L}^N)$.
 \end{theorem}

\subsection{Weighted graphs}
 Let $[V(G), d_G, (m^G_x)]$ be the  metric random walk space associated with a locally finite weighted   connected discrete
  graph $G = (V(G), E(G))$, as described in Example \ref{graphs1}. Let $\Omega \subset V(G)$ be a finite set. It is easy to see  (see \cite{MST2})
  that $[\Omega_{m^G}, d_G, m^G, \nu_G]$
  satisfies Poincar\'{e}'s inequality~\eqref{agost004}.  Therefore, if we consider the problem
\begin{equation}\label{Neumann4NG}
\left\{ \begin{array}{ll} u_t(t,x) = \displaystyle \frac{1}{d_x} \sum_{y \in V} w_{x,y} \vert u(y)-u(x)\vert^{p-2}(u(y) - u(x)), \quad
   &x\in  \Omega,\ 0<t<T, \\ \\ - \displaystyle \frac{1}{d_x} \sum_{y \in V} w_{x,y} \vert u(y)-u(x)\vert^{p-2}(u(y) - u(x)) =
  0, \quad   &x\in V \setminus \Omega, \  0<t<T, \\ \\ u(x,0) = u_0(x),    &x\in \Omega, \end{array} \right.
\end{equation}
we can apply Theorem \ref{nsth01S}, to get the following existence and uniqueness result.

\begin{theorem}\label{nsth01JG}  For any
$u_0\in L^{p'}(\Omega, \nu_G)$  there exists a unique strong solution $ u(t,x)$ of Problem~\eqref{Neumann4NG}.
 Moreover, for any $q\geq p'$ and  $u_{0i}\in L^q(\Omega,\nu_G)$,  $i=1,2$, we have the following contraction principle for the corresponding strong solutions $u_i$:
$$  \Vert (u_1(t,.)-u_2(t,.))^+\Vert_{L^q(\Omega,\nu_G)}\le \Vert (u_{0,1}-u_{0,2})^+\Vert_{L^q(\Omega,\nu_G)}
 \quad \hbox{for any \  \ $0\le t< T$.}$$
\end{theorem}

Consider now the problem
\begin{equation}\label{GraphNeumann4NValdi}
\left\{ \begin{array}{ll} u_t(t,x) = \displaystyle \frac{1}{d_x} \sum_{y \in V} w_{x,y} \a_p(x,y,u(y)-u(x)), \quad
   &x\in  \Omega,\ 0<t<T,
 \\ \\ - \displaystyle\frac{1}{d_x} \sum_{y \in \Omega} w_{x,y} \a_p(x,y,u(y)-u(x))  = \varphi(x), \quad   &x\in V \setminus \Omega, \  0<t<T, \\ \\ u(x,0) = u_0(x),    &x\in \Omega. \end{array} \right.
\end{equation}
Applying Theorem \ref{nsth01N}, we get the following existence and uniqueness result.

\begin{theorem}\label{euValdi2}
 Let $\varphi\in L^{\infty}(V \setminus \Omega, \nu_G)$. For every $u_0 \in L^\infty (\Omega,\nu_G)$ there exists a unique strong solution of problem \eqref{GraphNeumann4NValdi}.
 \end{theorem}

\bigskip

\noindent {\bf Acknowledgment.}The authors have been partially supported  by the Spanish MICIU and FEDER, project PGC2018-094775-B-100.
 The second author was also supported by the Spanish MICIU under grant BES-2016-079019, which is also supported by the European FSE.

\end{document}